\documentclass[12pt]{article}

\usepackage{fullpage,amsfonts,amsmath,amsthm,amssymb,mathrsfs,graphicx}

\usepackage[usenames]{color}

\usepackage[margin=1in]{geometry}

\begin{document}

\newcommand{\mm}[1]{{\color{blue}{#1}}}

\definecolor{pink}{rgb}{1,0.08,0.45}
\newcommand{\refc}[1]{{\color{pink}{#1}}}

\newcommand{\ja}[1]{{\color{blue}{#1}}}
\newcommand{\jb}[1]{{\color{Green}{#1}}}
\newcommand{\jc}[1]{{\color{Orange}{#1}}}
\newcommand{\jd}[1]{{\color{Purple}{#1}}}
\newcommand{\je}[1]{{\color{OliveGreen}{#1}}}

\def\dd{r}
\def\ld{{\widehat L}}
\def\Gcal{\mathcal G}
\def\eps{\epsilon}
\def\eqn{\eqref}
\def\wtau{\widehat \tau}
\def\ttau{\widetilde \tau}
\def\Q{\mathcal Q}
\def\hG{\widehat G}
\def\C{\mathcal C}
\def\ds{{\mathcal D}}
\def\hL{{\widehat L}}
\def\br{br}
\def\M{{\cal M}}
\def\H{{\cal H}}
\def\bell{\ensuremath{\boldsymbol\ell}}
\def\Bin{{\bf Bin}}
\def\hit{{\mathscr H}}
\def\hzeta{{\widehat \zeta}}
\def\htheta{{\widehat \theta}}
\def\hD{{\widehat D}}
\def\hN{{\widehat N}}
\def\hp{{\widehat p}}
\def\htau{{\widehat \tau}}
\def\vs{{\varsigma}}
\def\hH{{\widehat H}}
\def\deltac{{\xi}}
\def\hxi{{\widehat \xi}}

\newcommand{\be}{\begin{equation}}
\newcommand{\ee}{\end{equation}}
\newcommand{\bea}{\begin{eqnarray}}
\newcommand{\eea}{\end{eqnarray}}
\newcommand{\bean}{\begin{eqnarray*}}
\newcommand{\eean}{\end{eqnarray*}}
\newcommand{\non}{\nonumber}
\newcommand{\no}{\noindent}
\newcommand\floor[1]{{\lfloor #1 \rfloor}}
\newcommand\ceil[1]{{\lceil #1 \rceil}}
\newcommand{\jt}[1]{{\color{red} #1 }}
\newcommand{\jn}[1]{{\color{pink}\ #1 }}
\newcommand{\remove}[1]{}
\newcommand{\lab}[1]{\label{#1}\ }

\def\a{\alpha}
\def\b{\beta}
\def\d{\delta}
\def\D{\Delta}
\def\e{\epsilon}
\def\f{\phi}
\def\F{\Phi}
\def\g{\gamma}
\def\G{\Gamma}
\def\k{\eta}
\def\K{\eta}
\def\z{\zeta}
\def\th{\theta}
\def\Th{\Theta}
\def\l{\lambda}
\def\la{\lambda}
\def\La{\Lambda}
\def\m{\mu}
\def\n{\nu}
\def\p{\pi}
\def\P{\Pi}
\def\r{\rho}
\def\R{\Rho}
\def\s{\sigma}
\def\S{\Sigma}
\def\t{\tau}
\def\om{\omega}
\def\Om{\Omega}
\def\smallo{{\rm o}}
\def\bigo{{\rm O}}
\def\to{\rightarrow}
\def\E{{\bf Exp}}
\def\ex{{\mathbb E}}
\def\cd{{\cal D}}
\def\rme{{\rm e}}
\def\hf{{1\over2}}
\def\R{{\bf  R}}
\def\cala{{\cal A}}
\def\cale{{\cal E}}
\def\call{{\cal L}}
\def\cald{{\cal D}}
\def\calz{{\cal Z}}
\def\calf{{\cal F}}
\def\Fscr{{\cal F}}
\def\cc{{\cal C}}
\def\calc{{\cal C}}
\def\calh{{\cal H}}
\def\calk{{\cal K}}
\def\cals{{\cal S}}
\def\calr{{\cal R}}
\def\calt{{\cal T}}
\def\msq{{\mathscr Q}}
\def\bk{\backslash}

\def\heps{{\widehat\epsilon}}

\newcommand{\Exp}{\mbox{\bf Exp}}
\newcommand{\var}{\mbox{\bf Var}}
\newcommand{\pr}{\mbox{\bf Pr}}

\def \S{\mathcal S}
\def \SS{\mathscr C}

\newtheorem{lemma}{Lemma}
\newtheorem{theorem}{Theorem}
\newtheorem{corollary}[lemma]{Corollary}
\newtheorem{claim}[lemma]{Claim}
\newtheorem{remark}[lemma]{Remark}
\newtheorem{proposition}[lemma]{Proposition}
\newtheorem{observation}[lemma]{Observation}
\theoremstyle{definition}
\newtheorem{definition}[lemma]{Definition}

\newcommand{\limninf}{\lim_{n \rightarrow \infty}}
\newcommand{\proofstart}{{\bf Proof\hspace{2em}}}
\newcommand{\tset}{\mbox{$\cal T$}}
\newcommand{\proofend}{\hspace*{\fill}\mbox{$\Box$}}
\newcommand{\bfm}[1]{\mbox{\boldmath $#1$}}
\newcommand{\reals}{\mbox{\bfm{R}}}
\newcommand{\expect}{\mbox{\bf Exp}}
\newcommand{\he}{\hat{\e}}
\newcommand{\card}[1]{\mbox{$|#1|$}}
\newcommand{\rup}[1]{\mbox{$\lceil{ #1}\rceil$}}
\newcommand{\rdn}[1]{\mbox{$\lfloor{ #1}\rfloor$}}
\newcommand{\ov}[1]{\mbox{$\overline{ #1}$}}
\newcommand{\inv}[1]{\frac{1}{#1} }
\newcommand{\imax}{I_{\rm max}}

\newcommand{\whp}{w.h.p.}
\newcommand{\aas}{a.a.s.\ }

\date{\empty}

\title{The stripping process can be slow: part II}
\author{Pu Gao\thanks{Research supported by NSERC. 
This work started when the author was affiliated with University of Toronto and the research was then supported by NSERC PDF.}\\
University of Waterloo\\
p3gao@uwaterloo.ca
}

\maketitle

\begin{abstract}

This paper is a continuation of the previous results on the stripping number of a random uniform hypergraph, and the maximum depth over all non-$k$-core vertices. The previous results focus on the supercritical case, whereas this work analyses these parameters in the subcritical regime and inside the critical window.

\end{abstract}

\section{Introduction}

Given a hypergraph $H$ and a positive integer $k$, the {\em parallel $k$-stripping process} on $H$ is the sequence $H_0, H_1, H_2, \ldots$ such that $H_0=H$ and for every $i\ge 1$, $H_i$ is obtained form $H_{i-1}$ by removing all vertices with degree less than $k$ in $H_{i-1}$ together with their incident hyperedges. The process terminates with the $k$-core of $H$: the maximum subgraph of $H$ with minimum degree at least $k$. Note that the $k$-core of $H$ can be empty. Let $s_k(H)$ denote the number of iterations this process takes and we call $s_k(H)$ the $k$-stripping number of $H$. As $k$ is fixed in this paper, we often drop $k$ from the above notation.

We will study $s(\H_r(n,m))$, where $\H_r(n,m)$ is a uniformly random hypergraph on $n$ vertices and $m$ hyperedges, each of size $r$. The only interesting range of $m$ for this study is $m=\Theta(n)$. For $m$ in other ranges the stripping number can be easily estimated with little effort. We write $m=cn$ throughout this paper, where $c$ is bounded from both above and below by some absolute positive constants.   Another closely related random hypergraph model is $\H_r(n,p)$ where each hyperedge in $\binom{[n]}{r}$ appears independently with probability $p$. By conditioning on the number of hyperedges in $\H_r(n,p)$,   properties holding asymptotically almost surely (a.a.s.) in $\H_r(n,cn)$ usually translate immediately to $\H_r(n,r!c/n^{r-1})$.  

The stripping number of $\H_r(n,cn)$ is known to be small if $c$ is a constant and is not equal to $c_{r,k}$, the $k$-core emergence threshold (given in~\eqn{krthreshold}). It was proved~\cite{amxor} that the stripping number is $O(\log n)$ in this case. For $c>c_{r,k}+\eps$, this bound is proved to be tight~\cite{jmt}. For $c<c_{r,k}-\eps$, an improved upper bound $O(\log\log n)$ is given in~\cite{jmt,g2}. However, the stripping process can get very slow as $c\to c_{r,k}$. It was shown~\cite{gm} that if $c=c_{r,k}+n^{-\d}$ ($0<\d<1/2$) then the stripping number becomes $\Theta(n^{\d/2}\log n)$, whereas if $c=c_{r,k}-n^{-\d}$ then the stripping number is bounded by $\Omega(n^{\d/2})$ from below. The formal statement is as follows.

\begin{theorem}[\cite{gm}]\lab{thm:Stripsuper}
Let  $r,k\geq2, (r,k)\neq (2,2)$ be fixed. For any arbitrarily small $\eps>0$:
\begin{enumerate}
\item[(a)]  if $c\ge c_{r,k}+ n^{-1/2+\eps}$, then \aas $s(\calh_{\dd}(n,cn))=\Theta(\log n/\sqrt{\deltac})$, where $\deltac=|c-c_{r,k}|$.
\item[(b)]  if $|c-c_{\dd,k}|\le n^{-1/2+\eps}$, then \aas
$ s(\calh_{\dd}(n,cn))= \Omega(n^{1/4-\eps/2})$.
\item[(c)] if $c\le c_{r,k}-n^{-1/2+\eps}$, then a.a.s.\ $s(\calh_r(n,cn))=\Omega(1/\sqrt{\deltac})$, where $\deltac=|c-c_{r,k}|$.
\end{enumerate}
\end{theorem}

One of the main contributions in this paper is to prove that the lower bound in Theorem~\ref{thm:Stripsuper}(c) is almost tight. We will also provide an upper bound for $s(\H_r(n,cn))$ when $c$ is inside the critical window $|c-c_{\dd,k}|\le n^{-1/2+\eps}$.

A $k$-stripping sequence $v_1,v_2,\ldots$ is a sequence of vertices, which can be deleted from the hypergraph in the order of the sequence, such that each vertex has degree less than $k$ at the moment of its removal. 
 If $v$ is a vertex not contained in the $k$-core, then the depth of $v$ is the minimum integer $i$ such that there is a stripping sequence with $v_i=v$. In other words, the depth of $v$ is the minimum number of steps required to remove $v$ from $H$ among all stripping sequences.

For constant $c\neq c_{r,k}$, it is proved~\cite{amxor} that the maximum depth of the non-$k$-core vertices in $\calh_r(n,p=r!c/n^{r-1})$ is bounded by $O(\log n)$ and the same result easily translate to $\calh_r(n,cn)$. For $c=c_{r,k}+n^{-\d}$, the maximum depth is raised to $n^{\Theta(\delta)}$, proved by Molloy and the author~\cite{gm}, as follows.

\begin{theorem}[\cite{gm}]\lab{thm:depthsuper}
Assume $r,k\geq2, (r,k)\neq (2,2)$ are fixed. There are constants $a=a(r,k)$ and $b=b(r,k)$ such that for any $0< \xi_n<1/\log^{7}n$ and $c=c_{r,k}+\xi_n$,  a.a.s.\ the maximum depth of all non-$k$-core vertices of $\calh_{r}(n,cn)$ is between $\xi_n^{-a}$ and $\xi_n^{-b}$.
\end{theorem}
 In this paper, we will prove that the same statement holds for $c=c_{r,k}-\xi_n$.

Before proceeding to the statements of our main results, we briefly discuss the motivation of studying these two parameters, which has been addressed in~\cite{gm}. Several applications of the parallel stripping process was given in~\cite{jmt}, such as parity-check codes and hash-based sketches. However, a major motivation for studying the particular critical case where $c\to c_{r,k}$ is to investigate the solution clustering in random XORSAT. Research on the solution space of many random constraint satisfaction problems (CSPs) was started in statistical physics, and has received great attention in recent years in many broad areas such as physics, computer science, and combinatorics. Due to the discovery of the physicists, the solution space of a random CSP instance undergoes several phase transitions before its density reaches the satisfiability threshold. These phase transitions include clustering, variable freezing and condensation. See~\cite{gmarxiv} for a brief introduction, and the references in~\cite{gmarxiv} for the literature in this blossoming area. Understanding these phase transitions and the geometric properties of the solution space in each phase has been crucial in several recent achievements in the study of random CSPs, including solving the famous $k$-SAT conjecture (for large $k$)~\cite{DSS}. Clustering of random $r$-XORSAT was analysed independently in~\cite{CDMM} and~\cite{MRZ}, with some key arguments missing. The rigorous arguments determining the clustering threshold for random $r$-XORSAT were given independently in~\cite{amxor,ikkm}. Random $r$-XORSAT clustering coincides with the appearance of a non-empty 2-core of a random $r$-uniform hypergraph ($r\ge 3$). To bound the connectivity parameter of each cluster, the key arguments in~\cite{amxor} are to link this parameter to the maximum depth of all non-2-core vertices. For constant $c\neq c_{r,2}$, this parameter is bounded by $O(\log n)$~\cite{amxor}.  In order to charactersie how XORSAT-clusters are born, and how the cluster connectivity parameter transits around clustering, we need to estimate the maximum depth of the non-2-core vertices of $\H_r(n,cn)$ for $c\to c_{r,2}$, especially for $c=c_{r,2}+n^{-\d}$ and $c=c_{r,2}-n^{-\d}$ for some sufficiently small $\delta>0$. The birth of XORSAT-clusters will be studied in a following paper, whereas a preliminary version has been available in~\cite{gmarxiv} (for $c=c_{r,2}+n^{-\d}$).

All asymptotics in this paper refers to $n\to\infty$. For two sequences of real numbers $(f_n)$ and $(g_n)$, we say $f_n=O(g_n)$ if there is a constant $C>0$ such that $|f_n|\le C|g_n|$ for every $n\ge 1$. We write $f_n=o(g_n)$ if $\lim_{n\to\infty} f_n/g_n=0$;  $f_n=\Omega(g_n)$ if $f_n>0$ and $g_n=O(f_n)$. We use $f_n=\Theta(g_n)$ if $f_n>0$, $f_n=O(g_n)$  and $g_n=O(f_n)$.

\section{Main results}

The $k$-core emergence threshold was pursued by several authors~\cite{Chvatal,Luczak,Luczak2} before its determination, and was first determined by Pittel, Spencer and Wormald~\cite{psw} for random graphs ${\mathcal G}(n,m)$. This threshold was further determined in other random graph models and random hypergraphs~\cite{molloy,jhk,JL,cc}.  Recall that $c_{r,k}$ denotes the $k$-core emergence threshold of $\H_r(n,cn)$; then

\begin{equation}\lab{krthreshold}
c_{r,k}=\inf_{\mu > 0}
 \frac{\mu}{r\left[e^{-\mu}\sum_{i = {k-1}}^{\infty} \mu^i/i!\right]^{r-1}}
 \enspace.
 \end{equation}
Our main result of the stripping number of $\calh_{r}(n,cn)$ is the following. 
\begin{theorem}\lab{thm:Stripsub}
Let $r,k\geq2, (r,k)\neq (2,2)$ be fixed integers and $\eps>0$ be a constant.
\begin{enumerate}
\item[(a)]  If $c\le c_{r,k}- n^{-1/2+\eps}$, then \aas $s(\calh_{\dd}(n,cn))= O(\deltac^{-1/2}\log(1/\deltac)+\log\log n)$, where $\deltac=|c_{r,k}-c|$. \item[(b)]  If $|c-c_{\dd,k}|\le n^{-1/2+\eps}$, then \aas
$ s(\calh_{\dd}(n,cn))= O(n^{3/4+\eps})$.
\end{enumerate}
\end{theorem}
\no {\bf Remark}. Note that if $c_{r,k}-c$ is bounded below by a positive constant, then the upper bound becomes $O(\log\log n)$, which agrees with the upper bound in~\cite{jmt,g2}. If $c\le c_{r,k}-n^{-1/2+\eps}$, then the upper bound differs from the lower bound (c.f.~Theorem~\ref{thm:Stripsuper}(c)) by at most a constant factor of $\log n$.  For $\deltac=O(n^{-1/2+\eps})$, the upper bound does not match the existing lower bound (c.f.~Theorem~\ref{thm:Stripsuper}(b)).  The difficulty of obtaining tight bounds inside the critical window lies in the uncertainty of the existence of a non-empty $k$-core.  \smallskip

The next theorem bounds the maximum depth of the non-$k$-core vertices in $H_r(n,cn)$ for $c< c_{r,k}$.

\begin{theorem}\lab{thm:depthsub}
Let $r,k\geq2, (r,k)\neq (2,2)$ be fixed integers. There exist two constants $a=a(k,r)>0$ and $b=b(k,r)>0$ such that for any $0< \xi_n<1/\log^{7}n$ and $c=c_{r,k}-\xi_n$,  a.a.s.\ the maximum depth of the non-$k$-core vertices in $\calh_r(n,cn)$ is between $\xi_n^{-a}$ and $\xi_n^{-b}$.
\end{theorem}

\no {\bf Remark}. If we write $\xi_n=n^{-\delta}$, Theorem~\ref{thm:depthsub} states that the maximum depth of the non-$k$-core vertices is $n^{\Theta(\d)}$. Since the depth of each non-$k$-core vertex is bounded trivially by $n$, the upper bounds in Theorems~\ref{thm:depthsub} and~\ref{thm:depthsuper} are non-trivial only for $\xi=n^{-\d}$ where $\d>0$ is sufficiently small. Same as in~\cite{gm}, the condition $\deltac_n<1/\log^7 n$ can possibly be weakened and we did not try to optimise the power of the logarithm. This is because the most interesting applications of this theorem, e.g.\ XORSAT clustering, are for $c=c_{r,k}-n^{-\d}$ with some small constant $\d>0$.\smallskip

Our analysis focuses mainly on iterations of the parallel stripping process where the number of vertices is very close to some critical value. To describe this, we start by defining
\begin{eqnarray*}
f_t(\la)&=&e^{-\la}\sum_{i\geq t}\frac{\la^i}{i!};\\
h(\mu)=h_{\dd,k}(\mu)&=&\frac{\mu}{f_k(\mu)^{\dd-1}}.
\end{eqnarray*}
Note that $f_t(\la)$ is the probability that a Poisson variable with mean $\la$ is at least $t$.
 Now for any $\dd,k\geq2, (\dd,k)\neq (2,2)$, we define
$\mu_{\dd,k}$ to be the value of $\mu$ that minimizes $h(\mu)$; i.e. the (unique) solution to:
\begin{equation}
c_{\dd,k}=h(\mu_{\dd,k})/r.\lab{murk}
\end{equation}
Define
\begin{eqnarray}
\a=\a_{\dd,k}&=&f_k(\mu_{\dd,k})\lab{alpha}\\
\b=\b_{\dd,k}&=&\frac{1}{\dd}\mu_{\dd,k} f_{k-1}(\mu_{\dd,k}). \lab{beta}
\end{eqnarray}

 For ease of notation, we drop most of the $\dd,k$ subscripts.
For any $c\ge c_{r,k}$, we define $\mu(c)$ to be the larger solution to
\[c= h(\mu)/r.\] Then, $\mu_{\dd,k}=\mu(c_{\dd,k})$.
Define
\begin{eqnarray*}
\a(c)&=&f_k(\mu(c)),\quad \b(c)=\frac{1}{r}\mu(c) f_{k-1}(\mu(c)).
\end{eqnarray*}

Let $\C_k(H)$ denote the $k$-core of $H$. The following result on the $k$-core emergence threshold can be easily deduced from~\cite{jhk} (see the discussion above~\cite[Lemma 7]{gm}).

\begin{theorem}\lab{tkim} Let $\dd,k\geq2, (\dd,k)\neq (2,2)$ be fixed and $\e>0$ be an arbitrary constant.
\begin{enumerate}
\item[(a)] If $c\le c_{\dd,k}-n^{-1/2+\eps}$, then \aas $\C_k(\calh_{\dd}(n,cn))$ is empty.
\item[(b)] If $c\ge c_{\dd,k}+n^{-1/2+\eps}$, then \aas $\C_k(\calh_{\dd}(n,cn))$ has $\a(c)n+O(n^{3/4})$ vertices and $\b(c)n+O(n^{3/4})$ hyperedges.
\end{enumerate}
\end{theorem}

When the parallel stripping process is applied to $\H_r(n,cn)$ where $c=c_{r,k}+o(1)$, the vertices are stripped off fast in the beginning; in each round there are a linear number of vertices being removed. This continues until the number of vertices gets close to $\a n$. Our analysis will start from there.

\section{The allocation-partition model}

It is not easy to analyse random processes such as the parallel stripping process when applied to $\H_r(n,m)$, due to the dependency between the hyperedges. Instead we consider the following alternative model, called the allocation-partition model (AP-model). Take $rm$ points and uniformly at random (u.a.r.) allocate the points into a set of $n$ bins. Then, take a uniform partition of the $rm$ points so that each part has size exactly $r$. We call the resulting probability space $AP_r(n,m)$ and each element in $AP_r(n,m)$ a configuration. Each configuration in $AP_r(n,m)$ corresponds to a multi-hypergraph by representing each bin as a vertex and each part in the partition as a hyperedge. In this paper we call bins as vertices for simplicity. Each part in the partition is an $r$-tuple of points, which we may call a hyperedge when there is no confusion. The {\em degree} of a vertex $u$ in a configuration is the number of points that $u$ contains. 
Note that the AP-model is similar to the configuration model of  Bollob\'as~\cite{bb}, except that in the configuration model, the degree sequence is specified initially whereas in the AP-model, the degree sequence is a random variable determined by the allocation of the points into the bins. 
A simple counting argument shows that each hypergraph in $\H_r(n,m)$ corresponds to the same number of configurations in $AP_r(n,m)$. Therefore,  $AP_r(n,m)$ generates the hypergraphs in $\H_r(n,m)$ uniformly by conditioning on the resulting hypergraph being simple. For $m=O(n)$, following a result by Chv\'{a}tal~\cite{Chvatal} that the probability that a configuration in $AP_r(n,m)$ corresponds to a simple hypergraph is bounded away from zero (the proof in~\cite{Chvatal} is for $r=2$ but easily extends to general $r\ge 2$), the following corollary allows one to translate a.a.s.\ properties of $AP_r(n,m)$ to $\H_r(n,m)$.




\begin{corollary}\lab{ccon}
If $m=O(n)$ and property $Q$ holds a.a.s.\ in $AP_r(n,m)$, then $Q$ holds a.a.s.\ in $\H_r(n,m)$.
\end{corollary}
Running the parallel stripping process on a configuration of the AP-model is a natural extension.
In the rest of the paper, we will study $s(H)$ and the maximum depth of the non-$k$-core vertices of $H$ for $H\in AP_r(n,cn)$. We note here that Theorem~\ref{tkim} holds for $AP_r(n,cn)$ as well, as claimed in~\cite{CW}.


\section{Proof of Theorem~\ref{thm:Stripsub}}\lab{sec:Stripsub}

We bound $s(AP_r(n,cn))$ in this section. Without loss of generality, we assume $\deltac:=|c-c_{r,k}|=o(1)$ throughout the paper as the case $\deltac=\Omega(1)$ has already been verified in previous works, e.g.~\cite{amxor,g2,jmt}.  Recall that $H_0, H_1, H_2,\ldots$ is the parallel stripping process with $H_0\in AP_r(n,cn)$.  Let $S_i$ ($i\ge 0$) denote the set of light vertices (vertices with degree less than $k$) in $H_i$, i.e.\ the set of vertices removed during the $(i+1)$-th iteration of the parallel stripping process. Thus, $S_i=V(H_{i})\setminus V(H_{i+1})$.

 \subsection{Proof outline}

 Rather than starting our analysis with $H_0$ in the parallel stripping process, we would start our analysis from some $H_i$ (or some configuration ``close to'' $H_i$), where $i$ is chosen so that the number of vertices in $H_i$ is very close to $\a n$. The choice of $i$ relies on a coupling of two random configurations $AP_r(n,cn)$ and $AP_r(n,c'n)$. Let $\deltac'=c'-c_{r,k}$. Throughout Section~\ref{sec:Stripsub}, we always choose $c'$ satisfying the following conditions (recall that $\deltac=|c-c_{r,k}|=o(1)$ is assumed):
\be
\deltac'=O(c'-c)=o(1),\ c'-c=o(\sqrt{\deltac'}),\  c'\ge c_{r,k}+n^{-1/2+\eps} \ \mbox{for some constant $\eps>0$}.\lab{condc}
\ee
We do not repeat this assumption in all statements of lemmas.
We will describe the coupling in Section~\ref{sec:coupling}. The value of $c'$ is chosen differently in the proofs of part (a) and part (b) of Theorem~\ref{thm:Stripsub}.

We use the coupling to start our analysis from some $H_i$, or more precisely some configuration $G_0$ (defined in Section~\ref{sec:coupling}) close to $H_i$, such that (a), the number of light vertices in $G_0$ is of order $n(c'-c)$; (b), $i=O(\log(1/\deltac')/\sqrt{\deltac'})$. In Section~\ref{sec:G0} we analyse properties of $G_0$. In Section~\ref{sec:a} we prove part (a). We will choose some $c'>c_{r,k}$ such that $\deltac$ and $\deltac'$ are of the same asymptotic order. Then we specify two iterations $I_0$ and $I_1$ in the parallel process such that, by iteration $I_0$, $|S_i|$ decreases in each iteration; whereas starting from $I_1$, $|S_i|$ increases in each iteration until reaching a linear size; then $|S_i|$ keeps of linear size until the number of remaining vertices  in the configuration is at most $\sigma n$, for some small constant $\sigma>0$ ($\sigma$ is specified in Section~\ref{sec:sigma}). We will analyse closely the critical iterations from $I_0$ to $I_1$ during which the growth rate of $|S_i|$ changes from negative to positive, and we will bound $I_1-I_0$ by $O(1/\sqrt{\xi})$. Then, we will bound the growth rate of $|S_i|$ from below for each iteration after $I_1$, which allows us to bound the number of iterations needed until the number of  remaining vertices is at most $\sigma n$. In Section~\ref{sec:sigma}, we show that it takes $O(\log \log n)$ steps to strip off all vertices when there are at most $\sigma n$ vertices left.

In Section~\ref{sec:b}, we prove part (b). We will choose $c'=c_{r,k}+n^{-1/2+2\eps}$ in this case. We prove that within $n^{3/4+\eps}$ iterations, either the parallel stripping process has terminated with a non-empty $k$-core; or $|S_i|$ becomes reasonably large and it will grow in each iteration until reaching a linear size. In the latter case, with a similar argument as for part (a), we can bound the number of remaining iterations until all vertices are removed.

\subsection{Coupling}
\lab{sec:coupling}

We will couple two random configurations $(H',H)$ as follows. Let $H'\in AP_r(n,c'n)$ where $c'> c$ and $c'>c_{r,k}+n^{-\d'}$ for some constant $0<\d'<1/2$ (this ensures that a.a.s.\ $AP_r(n,c'n)$ has a non-empty $k$-core by Theorem~\ref{tkim}). Generate a random configuration $H$ by uniformly at random removing $(c'-c)n$ $r$-tuples in $H'$.  This resulting $H$ has the distribution $AP_r(n,cn)$ and moreover $H\subseteq H'$.  Run the parallel stripping process on $H'$ which yields a sequence $(H'_t)_{t\ge 0}$ with $H'_0=H'$.
Recall that $S'_i=V(H'_i)\setminus V(H'_{i+1})$. Let $B>0$ be a large constant to be specified later. Define
\be
\tau'(B)=\min\{t\ge B:\ |S'_t|\le n\deltac'\}.\lab{tauprime}
\ee
Now we define a coupled process $(\hH'_t,\hH_t)_{t\ge 0}$ as follows. Let $\hH'_0=H'$ and $\hH_0=H$. For each $1\le t\le \tau'(B)$, define $\hH'_t=H'_t$ and $\hH_t$ to be the configuration obtained from $\hH_{t-1}$ by removing all vertices in $V(H'_{t-1})\setminus V(H'_{t})$ together with their incident $r$-tuples. Since $\hH_0\subseteq \hH'_0=H'_0$, the set of vertices removed in each step of $(\hH_t)_{t=0}^{\tau'(B)}$ has degree less than $k$. Hence, $(\hH_t)_{t=0}^{\tau'(B)}$ can be viewed as a slowed-down version of the parallel stripping process on $\hH_0=H$.  
In other words, if we let $H''$ denote the configuration obtained by removing all light vertices  in $\hH_{\tau'(B)}$, then $H''$ must be a configuration $H_i$ that appears in the parallel stripping process $H_0,H_1,\ldots$ for some $i\le \tau'(B)+1$. Then,
\be
s(H)\le \tau'(B)+1+s(H'')\le \tau'(B)+s(\hH_{\tau'(B)})+1,\lab{couple}
\ee
since $s(H'')\le s(\hH_{\tau'(B)})$ as $H''\subseteq\hH_{\tau'(B)}$ by definition.  
It only remains to specify $B$, and to bound $\tau'(B)$ and $s(\hH_{\tau'(B)})$.
\smallskip

\no {\bf Specifying $B$ and bounding $\tau'(B)$}\smallskip

The behaviour of $(|S'_i|)_{i\ge 0}$ has been well studied in the prior paper~\cite{gm} (recall that for $H'$ we have $c'\ge c_{r,k}+ n^{-1/2+\eps}$, assumed in~\eqn{condc}). We cite here the relevant parts of~\cite[Lemma 49(a--c)]{gm}, which enables us to bound $\tau'(B)$, and will be useful for later use.
\begin{lemma}\lab{lem:Si}
There exist positive constants $B$, $Y_1$, $Y_2$ and $Z_1$ dependent only on $r$ and $k$, such that a.a.s. for
every $i\ge B$ with $|S'_i|\ge (1/\deltac')\log^2 n$, we have
\begin{enumerate}
\item[(a)] if $|S'_i|\ge n\deltac'$ then $(1-Y_1\sqrt{|S'_i|/n})|S'_i|\le |S'_{i+1}|\le (1-Y_2\sqrt{|S'_i|/n})|S'_i|$;
\item[(b)] if $|S'_i|<n\deltac'$ then $(1-Y_1\sqrt{\deltac'})|S'_i|\le |S'_{i+1}|\le (1-Y_2\sqrt{\deltac'})|S'_i|$;
\item[(c)] $\sum_{j\ge i}|S'_j| \le (Z_1/\sqrt{\deltac'}) |S'_i| $.
\end{enumerate}
\end{lemma}

Let $B$ be a constant chosen to satisfy Lemma~\ref{lem:Si}; this completes the definition of $\tau'(B)$ in~\eqn{tauprime}. Since $|S'_i|>n\deltac'$ for all $B\le i\le \tau'(B)-1$, recursively applying Lemma~\ref{lem:Si}(a), we have
\[
n\deltac'<|S'_{\tau'(B)-1}|\le (1-Y_2\sqrt{\deltac'})^{\tau'(B)-1-B} |S'_{B}|\le n\exp\Big(-Y_2\sqrt{\deltac'}(\tau'(B)-1-B)\Big). 
\]
This immediately yields that
\be
\tau'(B)=O\left(\frac{\log(1/\deltac')}{\sqrt{\deltac'}}\right). \lab{tauprimeBound}
\ee

By~\eqn{couple}, It only remains to bound $s(\hH_{\tau'(B)})$. 
By Lemma~\ref{lem:Si}(c) with $i=\tau'(B)$, it follows that 
\be
|H'_{\tau'(B)}|-|\C_k(H')|=O(n\deltac'/\sqrt{\deltac'})=O(n\sqrt{\deltac'}).\lab{G0tocore}
\ee

\no {\bf Generation of $\hH_{\tau'(B)}$}\smallskip

Note that $H$ is generated from $H'$ by removing u.a.r.\ $(c'-c)n$ $r$-tuples in $H'$. Let ${\mathcal E}=E(H')\setminus E(H)$ and let $X$ denote the number of $r$-tuples in ${\mathcal E}$ that are in $\hH'_{\tau'(B)}=H'_{\tau'(B)}$. Then, these $X$ $r$-tuples are uniformly distributed among all the $r$-tuples in $H'_{\tau'(B)}$. Hence,  $\hH_{\tau'(B)}$ can be generated by removing $X$ $r$-tuples uniformly at random in $H'_{\tau'(B)}$.

We bound $X$ in the following lemma (the proof is deferred). 
\begin{lemma}\lab{lem:X}
A.a.s.\ $X=\Theta((c'-c)n)$.
\end{lemma}

To analyse $s(\hH_{\tau'(B)})$, it is more convenient to use a slowed-down version of the parallel stripping process, called SLOW-STRIP (which is also used in~\cite{gm}), defined as follows. It iteratively deletes an $r$-tuple incident with a light vertex. We will use a queue $\Q$ to store all light vertices.

\begin{tabbing}
{\bf SLOW-STRIP}\\
{Input:}  A configuration $G$.\\
Initialize: $t:=0,G_0:=G$, $\msq$ is the set of light vertices in $G$.\\
Whi\=le \= $\msq\neq\emptyset$:\\
\>Let $v$ be the next vertex in $\msq$.\\
\>If $v$ contains \= no points then remove $v$ from $G$ and from $\msq$.\\ \>Oth\=erwise:\\
\>\> Remove one point $x$ in $v$, and the $r-1$ points in the same part as $x$.\\
\>\> If some other vertex $u$ becomes light then add $u$ to the end of $\msq$.\\
\> $G_{t+1}$ is the resulting configuration; $t:=t+1$.
\end{tabbing}

Define
\be
G_0=\hH_{\tau'(B)} \lab{G0}
\ee
 and $G_t$ to be the resulting configuration after $t$ steps of SLOW-STRIP applied to $G_0$.
 We use $L(G)$ to denote the total degree of the light vertices in $G$. For simplicity, let $L_t=L(G_t)$. We will analyse $(L_t)_{t\ge 0}$, starting with $L_0$.
\begin{lemma}\lab{lem:L0} A.a.s.\ $L_0=\Theta((c'-c)n)$.
\end{lemma}

The proofs of Lemmas~\ref{lem:X} and~\ref{lem:L0} are simple and are deferred to Section~\ref{sec:G0} where we will have a close study of $G_0$. 
To close this subsection, we discuss several key parameters that will be used throughout the proof. Let $N(G)$ and $D(G)$ denote the number of heavy vertices and the total degree of the heavy vertices of a configuration $G$. Define
\be
\zeta(G)=\frac{D(G)}{N(G)}. \lab{zeta}
\ee
For simplicity, we use $N_t$, $D_t$ and $\zeta_t$ for $N(G_t)$, $D(G_t)$ and $\zeta(G_t)$.

When SLOW-STRIP is applied to a random configuration, the algorithm does not expose the partition initially. When a point $x$ in $v$ is chosen to be deleted in a step, SLOW-STRIP exposes the $r$-tuple that contains $x$ by u.a.r.\ choosing  another $r-1$ points from the remaining points. These $r$ points are removed in that step of SLOW-STRIP. The partition of the remaining points remains unexposed and is uniformly distributed. 

We use $G_t$ to denote the random configuration obtained after $t$ steps of SLOW-STRIP, even though the $r$-tuples in $G_t$ have not been exposed yet, nor the allocation of all of the points in $G_t$.  Let $\calf_t=(L_t,N_t,D_t)$. The only information exposed after step $t$ is $\calf_t$,  and the set of vertices in $\Q$ as well as the points they contain. By the definition of the AP-model, conditioning on the exposed information, $G_t$ is a random configuration obtained by uniformly allocating the $D_t$ points to the $N_t$ heavy vertices, subject to each of these vertices receiving at least $k$ points, and u.a.r.\ partitioning the $L_t+D_t$ points. 
\begin{definition}
Define $\bar p_t=\bar p_t(\calf_t)$ to be the probability that a given point is allocated to a bin containing exactly $k$ points, in a uniform allocation of $D_t$ points into $N_t$ bins, subject to each bin receiving at least $k$ points. 
\end{definition}
 If we choose u.a.r.\ a point $x$ from the heavy vertices of $G_t$, then $\bar p_t$ is the probability that $x$ is contained in a vertex with degree $k$.  Analysing $\bar p_t$ is a key part of the analysis. We briefly explain why this parameter is important.

In each step of SLOW-STRIP, one point in a given light vertex is removed, together with another $r-1$ points $u_1,\ldots,u_{r-1}$ u.a.r.\ chosen from all the remaining points. For each of these $r-1$ points, if it is contained in a vertex with degree $k$, then the removal of this point results in a new light vertex and $L_t$ will increase by $k-1$. If we conditional on that the point is in a heavy vertex, then this probability is approximately $\bar p_t$. 

  We will restrict our analysis for steps $t$ such that the number of heavy vertices in $G_t$ is at least $\sigma n$ for some constant $\sigma>0$ to be specified in Section~\ref{sec:sigma}. Let $B_t=L_t+D_t$; i.e.\ $B_t$ denotes the total degree of $G_t$. Then, we may assume that $B_t\ge k\sigma n$. At step $t+1$, for each point $u_i$ ($1\le i\le r-1$) that are removed in this step, the probability that
it lies in a light vertex is $L_t/B_t+O(1/n)$.
Then, if $L_t>0$,
\be
\ex(L_{t+1}-L_t\mid \calf_t)=-1+(r-1)\left(-\frac{L_t}{B_t}+\left(1-\frac{L_t}{B_t}\right)(k-1)\bar p_t\right)+O(n^{-1}),\lab{L-diff}
\ee
where $O(n^{-1})$ above accounts for errors from two cases: (a), when deleting $u_i$, the total degree of the light vertices is $L_t+O(1)$ and the total degree is $B_t+O(1)$; (b), more than one of $u_i$'s are contained in the same vertex that becomes light after step $t+1$.

It is convenient to define
\be
\theta_t=-1+(r-1)(k-1)\bar p_t.\lab{deftheta}
\ee
 Then~\eqn{L-diff} can be rewritten as
\be
\ex(L_{t+1}-L_t\mid \calf_t)=\theta_t-(\theta_t+r)\frac{L_t}{B_t}+O(n^{-1}).\lab{Ldiff}
\ee

\subsection{Useful lemmas from~\cite{gm} and other related works}

In this section, we state several lemmas in the literature that are useful for our analysis. Most of them have appeared in the prior work~\cite{gm}. In Section~\ref{sec:coupling} we defined $G_0$, which is close to $\C_k(AP_r(n,c'n))$, since $G_0$ is obtained by removing a small number of $r$-tuples in $H'_{\tau'(B)}$, whereas $H'_{\tau'(B)}$ is close to $\C_k(AP_r(n,c'n))$ by~\eqn{G0tocore}. 

We mentioned in Section~\ref{sec:coupling} that we will keep track of $\bar p_t$, or correspondingly $\theta_t$. Note that $\theta_t$ is a function of $\calf_t$. Next, we specify this function.

Given positive integers $n$, $m$ and $k\ge 0$ such that $m\ge kn$, define $Multi(n,m,k)$, the {\em truncated multinomial distribution}, to be the probability space consisting of integer vectors ${\bf X}=(X_1,\ldots,X_n)$ with domain ${\mathcal I}_k:=\{{\bf d}=(d_1,\ldots,d_n):\ \sum_{i=1}^n d_i=m,\ d_i\ge k,\ \forall i\in[n]\}$, such that for any ${\bf d}\in {\mathcal I}_k$,
$$
\pr({\bf X}={\bf d})=\frac{m!}{n^m \Psi}\prod_{i\in [n]}\frac{1}{d_i! }=\frac{\prod_{i\in[n]}1/d_i!}{\sum_{{\bf d}\in{\mathcal I}_k}\prod_{i\in[n]}1/d_i!},
$$
where
$$
\Psi=\sum_{{\bf d}\in {\mathcal I}_k}\frac{m!}{n^m}\prod_{i\in [n]}\frac{1}{ d_i!}.
$$
It is well known that the degree distribution of the heavy vertices of $G_t$, conditional on $\calf_t$,
follows $Multi(N_t,D_t,k)$ (see e.g.~\cite{CW} for more details).
The truncated multinomial variables can be well approximated by independent truncated Poisson random variables. We formalise this in the following proposition whose proof can be found in~\cite[Lemma 1]{CW}.
\begin{proposition}\lab{p:Poisson}
Given integers $k$, $N$ and $D$ with $D> kN$, assume ${\bf X}\sim Multi(N,D,k)$. For any $j\ge k$, let $\rho_j$ denote the proportion of components in $X$ that equals $j$. Then, for any $\eps>0$, with probability $1-o(N^{-1})$,
\begin{equation}
\rho_j=e^{-\la}\frac{\la^j}{f_k(\la)j!}+O(N^{-1/2+\eps}), \lab{PoissonApprox}
\end{equation}
where $\la$ satisfies $\la f_{k-1}(\la)/f_k(\la)=D/N$.
\end{proposition}


Let $\la(x)$ be the root of $\la f_{k-1}(\la)=x f_k(\la)$, and define

\be
\psi(x)=\frac{e^{-\la(x)}\la(x)^{k-1}}{f_{k-1}(\la(x))(k-1)!}. \lab{h}
\ee

Recall from~\eqn{zeta} that $\zeta_t$ is the average degree of the heavy vertices of $G_t$. By definition, $\bar p_t=k\rho_k/\zeta_t$, where $\rho_k$ is the proportion of vertices with degree $k$. Thus $\bar p_t$ is approximately $\psi(\zeta_t)$ by~\eqn{h} and Proposition~\ref{p:Poisson}. 
With some elementary calculations (e.g.\ see~\cite[Lemma 30]{gm}) it is easy to show that $\psi(x)$ is decreasing on $(k,\infty)$. 
Thus, we have the following lemma of the approximation of $\bar p_t$ and $\theta_t$.
\begin{lemma}\lab{l:monotone}
Assume $N_t=\Omega(n)$. With probability $1-o(n^{-1})$,
\bean
\bar p_t&=& (1+O(n^{-1/2}\log n))\psi(\zeta_t),\\ 
\theta_t &=&-1+(k-1)(r-1)\psi(\zeta_t)+O(n^{-1/2}\log n).
\eean
Moreover, $\psi(x)$ is a strictly decreasing function on $x>k$; i.e.\ $\psi'(x)<0$ for every $x>k$. 
\end{lemma}

In the analysis for $(L_t)_{t\ge 0}$ and several other sequences of parameters, we frequently apply the following lemma, which is a simple application of the Hoeffding-Azuma inequality. 
\begin{lemma}\lab{l:azuma} Let $a_n$ and $c_n\ge 0$ be real numbers and $(X_{n,i})_{i\ge 0}$ be random variables with respect to a random process $(G_{n,i})_{i\ge 0}$ such that
$$
\ex(X_{n,i+1}\mid G_{n,i})\le X_{n,i}+a_n,
$$
and $|X_{n,i+1}-X_{n,i}|\le c_n$,
for every $i\ge 0$ and all (sufficiently large) $n$. Then, for any real number $j\ge 0$,
$$
\pr(X_{n,t}-X_{n,0}\ge ta_n+j)\le \exp\left(-\frac{j^2}{2t(c_n+|a_n|)^2}\right).
$$
\end{lemma}
\begin{proof} Let $Y_{n,i}=X_{n,i}-ia_n$. Then,
$$
\ex(Y_{n,i+1}\mid Y_{n,i})=\ex(X_{n,i+1}\mid X_{n,i})-(i+1)a_n\le  X_{n,i}-ia_n = Y_{n,i}.
$$
Thus, $(Y_{n,i})_{0\le i\le t}$ is a supermartingale. Moreover, $|Y_{n,i+1}-Y_{n,i}|\le c_n+|a_n|$. By Hoeffding-Azuma's inequality,
$$
\pr(Y_{n,t}-Y_{n,0}\ge j)\le \exp\left(-\frac{j^2}{2t(c_n+|a_n|)^2}\right).
$$
This completes the proof of the lemma.
\end{proof}

Recall the definition of $\a$ and $\b$ in~\eqn{alpha} and~\eqn{beta}.
As we mentioned before, our analysis focuses on a range of $t$, such that $G_t$ contains $(\a+o(1))n$ vertices. Hence, many parameters of $G_t$ are very close to certain critical values. For instance, $\zeta_t$ will be very close
\be
\zeta:=\frac{r\b}{\a}.\lab{criticalZeta}
\ee
By Lemma~\ref{l:monotone} and~\eqn{criticalZeta},  $\bar p_t$ will be very close to
$
\bar p:=\psi(\zeta)$.
A simple calculation (see~\cite[Lemmas 24 and 25]{gm} for a detailed proof) leads to 
\be
\psi(\zeta)=1/(r-1)(k-1) \ \mbox{and therefore}\ \ \bar p=\frac{1}{(r-1)(k-1)}.\lab{criticalP}
\ee
A non-trivial inequality with $\zeta$ is stated below, which will be useful in our proof. The proof can be found in~\cite[Lemmas 34]{gm}. \be
k<\zeta<r(k-1).\lab{criticalZetaInequality}
\ee

In order to analyse $G_0=\hH_{\tau'(B)}$, which is obtained by removing a few $r$-tuples in $H'_{\tau'(B)}$, we need information on how much $\zeta(H'_{\tau'(B)})$ deviates from $\zeta$.  The following lemma, proved in~\cite[Corollary 9]{gm}, estimates how much $\zeta(\C_k(H'))$ deviates from $\zeta$.

\begin{lemma}\lab{lcoresize}
For fixed $\dd,k\geq2, (\dd,k)\neq (2,2)$, there exist {three positive constants} $K_1=K_1(r,k),K_2=K_2(r,k)$ and $K_3=K_3(r,k)$ such that: if $c\ge c_{\dd,k}+ n^{-\d}$ for some constant $0<\d<1/2$ and $\deltac:=c-c_{r,k}=o(1)$, then a.a.s.\ $\C_k(AP_{\dd}(n,cn))$ has
$\a n +K_1 \sqrt{\deltac}n +O(\xi n+n^{3/4})$ vertices,
$\b n +K_2\sqrt{\deltac}n+O(\xi n+n^{{3/4}})$ $r$-tuples, and average degree $r\b/\a+K_3\sqrt{\deltac}+O(\deltac+n^{-1/4})$.
\end{lemma}
 By Lemma~\ref{lcoresize}, a.a.s.\
                $\zeta(\C_k(H'))=\zeta+\Theta(\sqrt{\deltac'})$. Then, by~\cite[Lemma 52(a)]{gm}, which states that a.a.s.\ $\zeta(H'_{\tau'(B)})\ge \zeta(\C_k(H'))+O(\log n/n)\ge\zeta+\Theta(\sqrt{\deltac'})$. On the other hand, 
by~\eqn{G0tocore}, 
$H'_{\tau'(B)}$ differs from $\C_k(H')$ by $O(n\sqrt{\deltac'})$ vertices, which affects the average degree of the heavy vertices (there are a.a.s.\ $\Omega(n)$ of them) by $O(\sqrt{\deltac'})$. So we must have 
$\zeta(H'_{\tau'(B)})= \zeta(\C_k(H'))+O(\sqrt{\deltac'})=\zeta+O(\sqrt{\deltac'})$. It follows then that
\be
\zeta(H'_{\tau'(B)})=\zeta+\Theta(\sqrt{\deltac'}).\lab{zeta0bound}
\ee

\subsection{Properties of $G_0$} \lab{sec:G0}

Recall that $H'\in AP_r(n,c'n)$. By Theorem~\ref{tkim}, a.a.s.\ $\C_k(H')$ contains approximately $\a n$ vertices and $\b n$ $r$-tuples.  We first prove Lemma~\ref{lem:X}.\smallskip

\no {\bf Proof of Lemma~\ref{lem:X}.\ }
Let ${\mathcal E}_1$ denote the set of $r$-tuples in $H'_{\tau'(B)}$ and ${\mathcal E}_2=E(H')\setminus {\mathcal E}_1$. Now $H$ is obtained by u.a.r.\ removing $(c'-c)n$ $r$-tuples from $E(H')={\mathcal E}_1\cup {\mathcal E}_2$. By definition, $X$ is the number of these deleted $r$-tuples that were in ${\mathcal E}_1$. By Theorem~\ref{tkim}, a.a.s.\ $|{\mathcal E}_1|=\Omega(n)$. Conditional on that, $X$ stochastically dominates $\Bin((c'-c)n, C)$ for some constant $0<C<1$. So a.a.s.\ $X=\Theta((c'-c)n)$, as $(c'-c)n\to\infty$ by~\eqn{condc}. \qed\smallskip

\remove{**************
**********************

We condition on the value of $N'$ and $M'$ as the number of vertices and the number of edges in $\C_k(H')$ and condition on the value of $X$; then $G_0$ can be generated by first generate a random $k$-core $G'$ with $N'$ vertices and $M'$ edges using the allocation-configuration model described as above and then remove u.a.r.\ $X$ edges from $G'$. Conditional on $G'$ being simple, $G_0$ has the correct distribution. It is easy to verify that a.a.s.\ the degree sequence of $G'$ (and $G_0$) is nice; if we conditional on the degree sequence of $G'$, then $G'$ is distributed as a random graph drawn from the configuration model, conditioned to a simple graph. By Corollary~\ref{ccon}, we may assume that $G'$ is drawn from the configuration model (conditioned to a given nice degree sequence) and prove a.a.s.\ properties of $G'$.

***********************
***********************
}

Recall that $G_0=\hH_{\tau'(B)}$, which is obtained by u.a.r.\ removing $X$ $r$-tuples from $H'_{\tau'(B)}$. Equivalently, $G_0$ is the configuration obtained by u.a.r.\ removing $Xr$ points from $\hH_{\tau'(B)}$, whereas the remaining vertices are uniformly partitioned. The following proposition follows by the uniformity of the allocation of points to bins in the AP-model. 

\begin{proposition}
Conditional on $N$, the number of heavy vertices in $G_0$, and $D$, the total degree of the heavy vertices, and the set $S$ of light vertices, $G_0$ is a random configuration uniformly drawn from the space where $D$ points are u.a.r.\ allocated to $N$ bins subject to each bin receiving at least $k$ points, and all points in the $N+|S|$ bins are uniformly partitioned into parts with size $r$.
\end{proposition}

Now we bound $L_0$, the total degree of the light vertices in $G_0$.\smallskip

\no {\bf Proof of Lemma~\ref{lem:L0}.\ } By the construction of $G_0$, 
\be
L_0= O(|S'_{\tau'(B)}|+X).\lab{L0}
\ee
 By~\eqn{tauprime} and Lemma~\ref{lem:Si}(a), $|S'_{\tau'(B)}|=\Theta(n\deltac')$. 
By Lemma~\ref{lem:X}, a.a.s.\ $X=\Theta((c'-c)n)$, and thus a.a.s.\ $L_0=O((c'-c)n)$ by~\eqn{L0} and~\eqn{condc}. It remains to prove the lower bound. 


Remove the $X$ $r$-tuples one by one. By Proposition~\ref{p:Poisson}, a.a.s.\ in every step the number of vertices with degree $k$ is $\Theta(n)$. So there is a constant $\gamma_0>0$ such that for each $r$-tuple that is removed, the probability that it is incident with a vertex with degree $k$ is at least $\gamma_0$. Hence, the number of vertices becoming light after removing $X$ random $r$-tuples stochastically dominates $\Bin(X,\gamma_0)$ and thus is a.a.s.\ $\Theta(X\gamma_0)=\Theta(n(c'-c))$. It only remains to show that not many of these new light vertices are created with degree zero. The expected number of vertices having degree drop from at least $k$ to zero is at most
\[
\sum_{j\ge 0}n \binom{rX}{j+1}\cdot O\left(\left(\frac{k+j}{n}\right)^{j+1}\right)=o((c'-c)n),
\]
where $n$ is an upper bound for the number of vertices with degree $k+j$, $\binom{rX}{j+1}$ is the number of ways to pick $j+1$ points from a set of $rX$ points, and  $O(((k+j)/n)^{j+1})$ is the probability that $j+1$ u.a.r.\ chosen points are contained in a given vertex with degree $k+j$. It follows then that a.a.s.\ $L_0=\Theta((c'-c)n)$. \qed \smallskip

Recall that $\theta_0=-1+(r-1)(k-1)\bar p_0$ by definition. We estimate $\theta_0$ in the following lemma.
\begin{lemma}\lab{lem:theta0}  There are constants $C_1>C_2>0$ such that a.a.s.\
$$
-C_1\sqrt{\deltac'}\le \theta_0\le -C_2\sqrt{\deltac'};\quad C_2\sqrt{\deltac'}\le \zeta_0-\zeta\le C_1\sqrt{\deltac'}
$$
\end{lemma}

\proof By~\eqn{zeta0bound}, $\zeta(H'_{\tau'(B)})-\zeta=\Theta(\sqrt{\deltac'})$. By Lemma~\ref{lem:X}, a.a.s.\ $X=\Theta((c'-c)n)$. It follows then that a.a.s., 
\[
\zeta_0=\zeta(H'_{\tau'(B)})(1+O(c'-c))=\zeta+\Theta(\sqrt{\deltac'}).
\]
This is because the number of heavy vertices and their total degree change by $O(X)=O((c'-c)n)$ from $H'_{\tau'(B)}$ to $G_0$, whereas a.a.s.\ the number of heavy vertices in $H'_{\tau'(B)}\supseteq\C_k(H')$ is $\Omega(n)$. The error $O(c'-c)$ is absorbed by $\Theta(\sqrt{\deltac'})$ as $c'-c=o(\sqrt{\deltac'})$ by~\eqn{condc}.

Now, by Lemma~\ref{l:monotone} and~\eqn{criticalP} and by taking the Taylor expansion of $\psi(x)$ at $x=\zeta$ (note that $\zeta>k$ by~\eqn{criticalZetaInequality}), a.a.s.\
\bean
\bar p_0&=&(1+O(n^{-1/2}\log n))\psi(\zeta_0)=(1+O(n^{-1/2}\log n))(\psi(\zeta)-\Theta(\sqrt{\deltac'}))\\
&=&\frac{1}{(r-1)(k-1)}-\Theta(\sqrt{\deltac'}),
\eean
where the error $O(n^{-1/2}\log n)$ is absorbed by $-\Theta(\sqrt{\deltac'})$, as $\deltac'\ge n^{-1/2+\eps}$.
Now, $\theta_0=-\Theta(\sqrt{\deltac'})$ follows by~\eqn{deftheta}.
 \qed

\subsection{Evolution of $\zeta_t$}
\lab{sec:M}

In order to analyse $(L_t)_{t\ge 0}$, we need to keep track of $\theta_t$ that appears in~\eqn{Ldiff}. Since $\theta_t$ is a function of $\zeta_t$ by Lemma~\ref{l:monotone}, we only need to investigate how $\zeta_t=D_t/N_t$ evolves in SLOW-STRIP.  Note that $\zeta_t$ would change in a step only if a point contained in a heavy vertex is removed. However, in a single step, the number of such points varies between zero and $k(r-1)$. In each step, aside from the point contained in the light vertex in the front of $\Q$, there are  $r-1$ other points removed, each u.a.r.\ chosen from all of the remaining points. We say an occurrence of event $\hit$ takes place if such a u.a.r.\ chosen point is contained in a heavy vertex.
In order to trim off the effect of  uncertainties, it is convenient to study an auxiliary process, defined below, in which exactly one occurrence of $\hit$ takes place in each step.


Let $\M$ be a configuration obtained by u.a.r.\ allocating ${\widehat D}$ points into ${\widehat N}$ bins, subject to each bin receiving at least $k$ points.  Let $\M_0=\M$. For every $t\ge 1$, $\M_t$ is obtained from $\M_{t-1}$ by removing a point u.a.r.\ chosen from all points; if it results in a bin containing less than $k$ points, remove that bin together with all points inside it. Note that if we choose ${\widehat N}$ and ${\widehat D}$ to be the number of heavy vertices and their total degree of $G_0$, then $(\M_t)_{t\ge 0}$ encodes the degree sequence of the heavy vertices of $(G_{t'})_{t'\ge 0}$.
To distinguish from parameters in $G_0,G_1,\ldots$, we add a hat to each corresponding parameter (like $\widehat{N_t}$, $\widehat{D_t}$, $\htheta_t$, $\hzeta_t$ etc.) for the sequence $\M_0,\M_1,\ldots$. We study parameters $\hzeta_t$ and $\htheta_t$ in this subsection, and will link them to $\zeta_t$ and $\theta_t$ in Section~\ref{sec:theta}.

Parameters ${\widehat D}$ and ${\widehat N}$ in $\M_0$ will eventually be chosen to coincide with $D_0$ and $N_0$ in $G_0$. Thus, by Lemma~\ref{lem:theta0}, we will assume the following conditions, which we omit in the statements of lemmas in this subsection.
\be\lab{hatzeta0}
\hzeta_0=\zeta+\Theta(\sqrt{\deltac'}), \ \mbox{where}\ n^{-1/2+\eps}\le \deltac'=o(1).
\ee
 By Lemma~\ref{l:monotone} and~\eqn{criticalP}, it is easy to see that \be
\lab{hattheta0}\htheta_0=-\Theta(\sqrt{\deltac'})
\ee
 which is negative (this is also indicated in Lemma~\ref{lem:theta0}).

Fix a constant $0<\sigma<1$ (the value of $\sigma$ will be determined in Section~\ref{sec:a}). Define $\tau_1=\tau_1(\sigma)$ to be the maximum integer $t$ such that $\hN_t\ge \sigma n$.
Obviously, for all $t\le \tau_1(\sigma)$,
\be
\hzeta_t-\hzeta_0=O(t/n),\quad \htheta_t-\htheta_0=O(t/n), \lab{theta-diff}
\ee
because ${\widehat N_t}$ and ${\widehat D_t}$ changes by $O(1)$ in each step.

\begin{lemma}\lab{lem:zeta} There exist constants $C_1>C_2>0$ such that 
\begin{enumerate}
\item[(a)]for all $0\le t\le \tau_1(\sigma)$,
\be
-C_1/n\le \ex(\hzeta_{t+1}-\hzeta_t\mid \M_t)\le -C_2/n;\lab{zeta-diff}
\ee
\item[(b)] a a.a.s.\ for all $\log^2 n\le t\le \tau_1(\sigma)$, $-C_1t/n\le \hzeta_t-\hzeta_0\le -C_2t/n$;
\item[(c)] a.a.s.\ for all $0\le t\le \tau_1(\sigma)$, $C_2t/n+O(n^{-1/2}\log n)\le \htheta_t-\htheta_0\le C_1t/n+O(n^{-1/2}\log n)$.
\end{enumerate}
\end{lemma}
\proof We first prove parts (a,b) for $0\le t\le \eps_0 n$, for some proper $\eps_0>0$.
 It was shown in~\cite[Lemma 36]{gm} (below eq.\ (30)) that~\eqn{zeta-diff} holds as long as $\hzeta_t<r(k-1)-\eps_1$ for some constant $\eps_1$ (and $C_1, C_2$ in (a) depends only on $\eps_1$). We have $\hzeta_0=\zeta+o(1)$ by~\eqn{hatzeta0} and immediately we have $\hzeta_t=\zeta+O(\eps_0)$ for all $0\le t\le \eps_0 n$ by~\eqn{theta-diff}. Therefore, by~\eqn{criticalZetaInequality}, there exist sufficiently small constants $\eps_1,\eps_0>0$ such that $\hzeta_t<r(k-1)-\eps_1$ for all $0\le t\le \eps_0 n$.

By Lemma~\ref{l:azuma} (with $a_n=-\Theta(1/n)$, $c_n=\Theta(1/n)$ and $j=ta_n$), a.a.s.\ for all $\log^2 n\le t\le \eps_0 n$, $\hzeta_t-\hzeta_0 \le -\Theta(t/n)$. Applying Lemma~\ref{l:azuma} again to $(-\hzeta_t)_{t\ge 0}$ (with $a_n=\Theta(1/n)$, $c_n=\Theta(1/n)$ and $j=-ta_n/2$), a.a.s.\ for all $\log^2 n\le t\le \eps_0 n$, $\hzeta_t-\hzeta_0 \ge -\Theta(t/n)$. It follows then that a.a.s.\ for all $\log^2 n\le t\le \eps_0 n$, $\hzeta_t-\hzeta_0 = -\Theta(t/n)$.

Next, we discuss $\eps_0 n\le t\le \tau_1(\sigma)$. We have shown that a.a.s.\ $\hzeta_{\eps_0 n}-\hzeta_0=-\Theta(\eps_0)$, i.e.\ 
$\hzeta_{\eps_0 n}<\hzeta_0$ and so the condition $\hzeta_t<r(k-1)-\eps_1$ holds for all $\eps_0 n\le t\le 2\eps_0 n$.
With the same argument, parts (a,b) hold for all $t$ in this range. Inductively, claims in parts (a,b) hold for all $t\le\tau_1(\sigma)$, as there are only $O(1/\eps_0)=O(1)$ inductive steps. 

For part (c), note that $\zeta_0>k$ since $\zeta_0=\zeta+O(\eps_0)$, and $\zeta>k$ by~\eqn{criticalZetaInequality} and that $\eps_0$ can be chosen sufficiently small. By Lemma~\ref{l:monotone} and by the union bound, we have that a.a.s.\ for all $0\le t\le \tau_1(\sigma)$,
\[
\htheta_t=-1+(k-1)(r-1)\psi(\zeta_t)+O(n^{-1/2}\log n).
\]
By Lemma~\ref{l:monotone} and part (b), and by taking the Taylor expansion of $\psi(x)$ at $x=\zeta_0$, we have
\[
\htheta_t=-1+(k-1)(r-1)(\psi(\zeta_0)+\Theta(t/n))+O(n^{-1/2}\log n)=\htheta_0+\Theta(t/n)+O(n^{-1/2}\log n),
\]
for all $\log^2 n\le t\le \tau_1(\sigma)$. The case $t<\log^2 n$ easily follows from~\eqn{theta-diff} by noting that $O(\log^2n/n)$ is absorbed by $O(n^{-1/2}\log n)$.\qed

\smallskip

\begin{corollary}\lab{cor:zeta1}
A.a.s.\ for all $0\le t\le \tau_1(\sigma)$,
$\htheta_t\ge 2\htheta_0$.
\end{corollary}
\proof  By Lemma~\ref{lem:zeta}(c), for all $t\le n^{1/2}\log^2 n$, we have $\htheta_t=\htheta_0+O(t/n+n^{-1/2}\log n)=\htheta_0+O(n^{-1/2}\log^2 n)$. By~\eqn{hattheta0} and~\eqn{hatzeta0}, $\htheta_0<0$ and $|\htheta_0|=\Omega(n^{-1/4})$. It follows then that  $\htheta_t\ge 2\htheta_0$ for all $t\le n^{1/2}\log^2 n$. By Lemma~\ref{lem:zeta}(c), there is a constant $C>0$ such that a.a.s.\ $\htheta_t\ge \htheta_0 +Ct/n\ge 2\htheta_0$ for all $n^{1/2}\log^2 n<t\le \tau_1(\sigma)$. \qed \smallskip

Define $t_0(K)=K n\sqrt{\deltac'}$, where $K>0$ is a large constant. Note that $\deltac'=o(1)$ implies $t_0(K)=o(n)$.
The following corollary states when $\htheta_t$ becomes positive (recall that $\htheta_0<0$).

\begin{corollary}\lab{cor:zeta} 
Assume that $K>0$ is a sufficiently large constant. A.a.s.\ there are constants $C_1,C_2>0$ such that for every $t_0(K)\le t\le \tau_1(\sigma)$, $C_1t/n\le \htheta_t\le C_2t/n$.
\end{corollary}

\begin{proof}
By Lemma~\ref{lem:zeta}(c) and~\eqn{hatzeta0}, and noting that $n^{-1/2}\log n=o(t_0(K)/n)$ by~\eqn{hatzeta0}, a.a.s.\ for any $t_0(K)\le t\le \tau_1(\sigma)$, $\htheta_t\ge -C\sqrt{\deltac'}+Yt/n$ for some constants $C,Y>0$.  Choosing $K\ge 2C/Y$ we have that a.a.s.\ $\htheta_t\ge (Y/2)t/n$ for all $t_0(K)\le t\le \tau_1(\sigma)$.
The upper bound of $\htheta_t$ follows by~\eqn{theta-diff} and the fact that $\htheta_0<0$.
\end{proof}

This immediately yields the following corollary.

\begin{corollary}\lab{cor2:zeta} For any constant $\eps>0$, a.a.s.\ $\htheta_t=\Omega(\eps)$  for all $\eps n\le t\le \tau_1(\sigma)$.
\end{corollary}

\subsection{Relating $\htheta_t$ to $\theta_t$}
\lab{sec:theta}

Recall that $G_0,G_1,\ldots$ is the process produced by SLOW-STRIP. 
To  analyse $\theta_t$ using $\htheta_t$ in Section~\ref{sec:M}, let $\widehat{N}$ and $\widehat{D}$ in the definition of $\M_0$ in Section~\ref{sec:M} take the same values as the corresponding parameters in $G_0$. Therefore, $\hzeta_0=\zeta_0$ and $\htheta_0=\theta_0$. By Lemma~\ref{lem:theta0}, a.a.s.\ $\theta_0=-\Theta(\sqrt{\deltac'})$ and $\zeta_0-\zeta=\Theta(\sqrt{\deltac'})$. This verifies the assumption~\eqn{hatzeta0}.

Corresponding to $\tau_1(\sigma)$, define
\be
\tau_2(\sigma)=\max\{t: N_t\ge \sigma n\}.\lab{tau2}
\ee

The following lemma allows us to establish a relation between $\theta_t$ and $\htheta_t$.

\begin{lemma} \lab{lem:hit}
There is a constant $C>0$: for any $\log^2 n\le t\le\tau_2(\sigma)$, a.a.s.\ the number of occurrences of $\hit$ by step $t$ is at least $Ct$.
\end{lemma}

\proof Let $h_t$ denote the number of occurrences of $\hit$ by step $t$. By our definition of $\tau_2(\sigma)$, the number of heavy vertices in every step is at least $\sigma n$ and thus, there is a constant $\sigma'>0$ (depending on $\sigma$) such that for every point that was u.a.r.\ chosen, the probability that it was contained in a heavy vertex is at least $\sigma'$. In every step, there are $r-1\ge 1$ such points being chosen. Hence, for all $t\le\tau_2(\sigma)$, we always have
\[
\ex(h_t\mid G_{t-1})\ge h_{t-1}+\sigma'.
\]
Our claim follows by applying Lemma~\ref{l:azuma} to $(-h_t)$ (with $a_n=-\sigma'$, $c_n=r-1$ and $j=(\sigma'/2)t$).\qed\smallskip

Noting that at most $r-1$ occurrences of $\hit$ can take place in a single step, this immediately gives the following corollary.

\begin{corollary}\lab{cor:hit}
There is a constant $0<C<1$ such that a.a.s.\ for all $\log^2 n\le t\le \tau_2(\sigma)$, $\theta_t=\htheta_{t'}$ for some $Ct<t'\le (r-1)t$.
\end{corollary}
Another corollary follows easily from Corollaries~\ref{cor:zeta} and~\ref{cor:hit} (recalling that $n\sqrt{\deltac'}=\Omega(n^{3/4})$ by~\eqn{condc}).
\begin{corollary}\lab{cor2:hit}
Let $K>0$ be a sufficiently large constant. There exist two constants $C_1,C_2>0$ such that a.a.s.\ for all $Kn\sqrt{\deltac'}\le t\le \tau_2(\sigma)$, $C_1t/n\le\theta_t\le C_2t/n$.
\end{corollary}

\subsection{Specifying $\sigma$}
\lab{sec:sigma}

The key lemma we use to specify the constant $\sigma$ is the following.

\begin{lemma}\lab{lem:small} Assume $c=c_{r,k}+o(1)$ and consider the parallel $k$-stripping process $H_0,H_1,H_2,\ldots$ with $H_0\in AP_r(n,cn)$.
There is a constant $\sigma_0>0$ such that a.a.s.\ if $H_i$ has at most $\sigma_0 n$ vertices for some $i>0$ then every component of $H_{i+1}$ contains $O(\log n)$ vertices.
\end{lemma}

 We will use the following two lemmas to prove Lemma~\ref{lem:small}. The first lemma is from~\cite[Lemma 7]{molloy}.
\begin{lemma}\lab{lem:constantRounds} Assume $c<c_{r,k}-\eps_0$ for some constant $\eps_0>0$ and consider the parallel $k$-stripping process $H_0,H_1,H_2,\ldots$ with $H_0\in AP_r(n,cn)$. Then, there are positive constants $(\gamma_i)_{i=0}^{\infty}$ with $\lim_{i\to\infty} \gamma_i=0$ such that, a.a.s.\ for every fixed $i\ge 0$, $|H_i|\sim \gamma_i n$.
\end{lemma}

Let $\rho_i(j)$ denotes the proportion of vertices with degree $j$ in $H_i$. Part (a) of the following lemma is from~\cite[Section 8]{amxor}, whereas part (b) is from~\cite[Theorem 1]{M6}.
\begin{lemma}\lab{lem:rho}
Assume $c<c_{r,k}-\eps_0$ for some constant $\eps_0>0$ and $H_0\in AP_r(n,cn)$. 
\begin{enumerate}
\item[(a)] For any constant $K>0$, there is a constant $I>0$ such that for all $i\ge I$,
\be
\rho_i(1)> K \sum_{j\ge 2}\big((k-1)j(j-1)-j\big)\rho_i(j). \lab{graphDeg}
\ee
\item[(b)] A.a.s.\ a random hypergraph with degree sequence satisfying~\eqn{graphDeg} for some constant $K>1$ has the property that each component has size $O(\log n)$. 
\end{enumerate}
\end{lemma}

Note that Lemmas~\ref{lem:constantRounds} and~\ref{lem:rho} were stated for $\H_r(n,r!c/n^{r-1})$ in the original papers. However, the proofs use the configuration model, which is the AP-model conditioned to ``typical'' degree sequences. Hence, these results also hold a.a.s.\ for $AP_r(n,cn)$.

 \begin{proof}[Proof of Lemma~\ref{lem:small}] 

Fix a small $\eps_0>0$ and let $c''=c_{r,k}-\eps_0$. Couple $H''_0\subseteq H_0$ in the same way as described in Section~\ref{sec:coupling} such that $H''_0\in AP_r(n,c''n)$; i.e.\ we generate $H''_0$ by removing u.a.r.\ $(c-c'')n$ $r$-tuples in $H_0$. Let ${\mathcal E}$ denote this set of $r$-tuples.
Applying Lemma~\ref{lem:rho} to $H''_0$, there is a sufficiently large constant $I$ such that~\eqn{graphDeg} is a.a.s.\ satisfied with $K=2$. Let $(\gamma_i)$ be the sequence in Lemma~\ref{lem:constantRounds} for $(H''_i)$, and let $\sigma_0:=\gamma_I/2$. 
 
   Consider the parallel stripping process $H_0, H_1,\ldots$. Let $i-1$ denote the first iteration after which there are at most $\sigma_0 n$ vertices remaining, if the process has not terminated by then. Define $i$ to be $n$ if no such iteration exists.
Note that $i$ is not (necessarily) a.a.s.\ bounded by a constant, since we have $c=c_{r,k}+o(1)$. 

Let $\widehat H$ be the random configuration obtained by removing all the $r$-tuples in ${\mathcal E}$ from $H_{i-1}$. Since $|{\mathcal E}|=(c-c'')n$,  the number of $r$-tuples removed is at most $(c-c'')n=O(\eps_0 n)$, and so the number of new light vertices created by the removal of the $r$-tuples in ${\mathcal E}$ is $O(\eps_0 n)$. Let $\widehat H'$ be the graph obtained from $\widehat H$ by removing all light vertices in $\widehat H$. Then, $H_i$ and $\widehat H'$ differ by at most $O(\eps_0 n)$ $r$-tuples. Moreover, as we have discussed before, $\widehat H'$ would have occurred in the parallel stripping process starting with $H''_0$, i.e.\ there is some $j$ such that $H''_j=\widehat H'$.


Since $H''_j=\widehat H'\subseteq H_{i-1}$, we have $|H''_j|\le \sigma_0 n$. Hence, we must have $j\ge I$, since $\sigma_0=\gamma_I/2$ by definition and $|H''_I|\sim \gamma_I n$ by Lemma~\ref{lem:constantRounds}.

By Lemma~\ref{lem:rho}(a) and the choice of $I$, a.a.s.\ $H''_j=\widehat H'$ satisfies~\eqn{graphDeg} with $K=2$. But $H_{i}$ and $H''_j=\widehat H'$ differ by only $O(\eps_0 n)$ $r$-tuples as we discussed before. It follows immediately that the degree sequence of $H_{i}$ will satisfy~\eqn{graphDeg} with $K=3/2$ as long as we choose $\eps_0>0$ sufficiently small. By Lemma~\ref{lem:rho}(b), a.a.s.\ every component in $H_{i}$ has size $O(\log n)$. 
\end{proof}

In the rest of the paper, we choose $\sigma$ to be the constant that satisfies Lemma~\ref{lem:small} and this fixes the definitions of $\tau_1(\sigma)$ in Section~\ref{sec:M} and $\tau_2(\sigma)$ in Section~\ref{sec:theta}.

Recall the definition of $G_0$ in Section~\ref{sec:coupling}. Let $\hG_0,\hG_1,\hG_2\ldots$ denote the parallel $k$-stripping process with $\hG_0=G_0$. With a slight abuse of notation, we use $(S_i)$ to denote the set of light vertices associated with $(\hG_i)$; i.e.\ $S_i=V(\hG_i)\setminus V(\hG_{i+1})$. Recall that $L(\hG_i)$ is the total degree of $S_i$ in $\hG_i$.

Define
\be
I_{\sigma}:=\max\{i: \ |\hG_i|\ge \sigma n\}. \lab{Isigma}
\ee
Then, $|\hG_{I_{\sigma}+1}|<\sigma n$ and so by Lemma~\ref{lem:small}, every component of $\hG_{I_{\sigma}+2}$ has size $O(\log n)$. It is easy to bound $s(\hG_{I_{\sigma}+2})$ by $O(\log\log n)$ (see the end of Section~\ref{sec:a}). Thus, in order to bound $s(G_0)$, it is sufficient to bound $I_{\sigma}$.
We often need to relate an iteration in the parallel stripping process to the step in SLOW-STRIP corresponding to the beginning of that iteration (recall that the process generated by SLOW-STRIP is denoted by $G_0,G_1,\ldots$ and $L_t$ denotes the degree of light vertices in $G_t$). To do so, we define $t(i)$ to be the step in SLOW-STRIP that the first vertex removed at the $i$-th iteration of the parallel process is pushed to the front of the queue $\Q$.
Therefore,
\be
L(\hG_i)=L_{t(i)}\ \mbox{for each $i\ge 0$}. \lab{t(i)}
\ee
 In particular, $L(\hG_0)=L_0$. 


\subsection{Subcritical: proof of Theorem~\ref{thm:Stripsub}(a)}
\lab{sec:a}

Now let $\eps>0$ be an arbitrary constant and we assume that $c\le c_{r,k}-n^{-1/2+\eps}$. Without loss of generality, we may assume that $\eps$ is sufficiently small. Let $0<\vs<1$ be a sufficiently small constant to be specified later; define $c'=c_{r,k}+\vs \deltac$; this fixes $c'$ introduced in Section~\ref{sec:coupling}. Clearly, all conditions in~\eqn{condc} are satisfied. 
By Lemma~\ref{lem:theta0}, a.a.s.\ $\theta_0=-\Theta(\sqrt{\vs\deltac})$ and $\zeta_0=\zeta+\Theta(\sqrt{\vs\deltac})$.

Recall that $L_0$ denotes the total degree of the light vertices in $G_0$. By Lemma~\ref{lem:L0}, a.a.s.\ $L_0=\Theta(n\deltac)$. Recalling~\eqn{couple} and~\eqn{G0}, our goal in this section is to bound $s(G_0)$.

\begin{lemma}\lab{lem:D} Suppose $\vs>0$ is sufficiently small. There is a constant $C>0$ and an integer $i<C/\sqrt{\deltac}$ such that
 a.a.s.\ $t(i)=\Theta(n\sqrt{\deltac})$, $\theta_{t(i)}= \Theta(\sqrt{\deltac})$ and $L_{t(i)}=\Theta(n\deltac)$.
 \end{lemma}

\proof 

We have $\theta_0\ge -Y_2\sqrt{\vs \deltac}$ for some constant $Y_2>0$ (note that $Y_2$ is independent of $\vs$), since $\deltac'=c'-c_{r,k}=\vs\deltac$. By Corollary~\ref{cor2:hit}, there is a constant $Y_3>0$ (depending on $Y_2$ and the constant in Corollary~\ref{cor2:hit}, but not on $\vs$) such that a.a.s.\ $\theta_{t_0}\ge \sqrt{\vs\deltac}$ for $t_0:=Y_3n\sqrt{\vs \deltac}$. 

By~\eqn{Ldiff}, if $L_t>0$,
\be
\ex(L_{t+1}-L_t\mid \calf_t)=\left(1-\frac{L_t}{B_t}\right)\theta_t-r\frac{L_t}{B_t}+O(n^{-1}).\lab{Ldiff2}
\ee
By Corollary~\ref{cor:zeta1} and~\ref{cor:hit}, a.a.s.\ $\theta_t\ge -2Y_2\sqrt{\vs \deltac}$ for all $0\le t\le t_0$.
  Hence, $(1-L_t/B_t)\theta_t\ge -2Y_2 \sqrt{\vs \deltac}$. Note that the number of vertices in $G_0$ is a.a.s.\ $(\alpha+o(1))n$ by the construction of $G_0$ and by Theorem~\ref{tkim}. This is true for all $G_t$ where $0\le t\le t_0$ as $t_0=o(n)$ by definition. If $L_t<x:=(\a Y_2/r)n \sqrt{\vs \deltac}$, then $rL_t/B_t\le Y_2\sqrt{\vs \deltac} $ since a.a.s.\ $B_t\ge k(\a+o(1))n\ge \a n$. So for all $0\le t\le t_0$ with $L_t>0$ and $L_t<x$,
\[
\ex(L_{t+1}\mid \calf_t)\ge  L_t - 3Y_2 \sqrt{\vs \deltac}.
\]
Applying Lemma~\ref{l:azuma} to $(-L_t)$ (with $a_n=3Y_2\sqrt{\vs\deltac}$, $c_n=kr$ and $j=3Y_2t\sqrt{\vs\delta}$), we immediately have that
a.a.s.\ for all $\log^2 n\le t\le t_0$, we have $L_{t}\ge L_0 - 6Y_2 t \sqrt{\vs \deltac}$. Hence, for all $0\le t\le t_0$, $L_t\ge L_0 - 6Y_2 t_0 \sqrt{\vs \deltac}= L_0-6Y_2Y_3 n\vs \deltac=\Omega(n\deltac)$ because $L_0=\Omega(n\deltac)$ and $\vs$ can be chosen sufficiently small (noting that the above inequality holds trivially for $t<\log^2 n$).
Thus, we have shown that there is a constant $Y_4>0$ such that a.a.s.\ $L_t\ge Y_4 n \deltac$ for all $0\le t\le t_0$.

Next, we prove that $L_t=O(n\deltac)$ for all $0\le t\le t_0$, which then will imply that $L_t=\Theta(n\deltac)$ for all $0\le t\le t_0$. By Lemma~\ref{lem:zeta}(c), Corollary~\ref{cor:hit} and~\eqn{theta-diff}, a.a.s.\ there  are constants $Y_5',Y_5>0$: $\theta_t\le \theta_0+ Y_5' t_0/n+O(n^{-1/2}\log n)\le Y_3Y_5\sqrt{\vs\deltac}$ for all $0\le t\le t_0$.
Then by~\eqn{Ldiff},
\[
\ex(L_{t+1}\mid \calf_t)\le  L_t +Y_3Y_5\sqrt{\vs\deltac}
\]
and thus by Lemma~\ref{l:azuma}, a.a.s.\ for all $\log^2 n\le t\le t_0$, $L_{t}\le L_0+ 2Y_3Y_5 t_0\sqrt{\vs\deltac}$ and therefore, a.a.s.\ for all $0\le t\le t_0$, $L_{t}=O(n\deltac)$ (the equation holds trivially for $t\le \log^2 n$).

Let $I_1$ (which will be the integer $i$ in the statement of this lemma) be the minimum integer that $t(I_1+1)\ge t_0$. We have shown that $L_t\ge Y_4n\xi$ for all $0\le t\le t_0$. So, the total degree of vertices in each $S_i$, $i\le I_1$, is at least $Y_4 n\deltac$. Thus, for each $i\le I_1$, the $i$-th iteration of the parallel stripping process is consist of at least $(Y_4/r) n\xi$ steps of SLOW-STRIP, since at most $r$ points contained in $S_i$ are deleted in every step of SLOW-STRIP. It follows then that
$I_1=O(t_0/ Y_4 n\deltac)$. 
Hence, $I_1=O(1/\sqrt{\deltac})$ as $t_0=Y_3n\sqrt{\vs\xi}$.

We have shown that a.a.s.\ $L_t=\Theta(n\deltac)$ for all $0\le t\le t_0$. We also have $t(I_1)<t_0$ by our definition of $I_1$. So, a.a.s.\ $L(\hG_{I_1})=L_{t(I_1)}=\Theta(n\deltac)$.

We have shown that $\theta_{t_0}=\Theta(\sqrt{\deltac})$. By the definition of $I_1$, $t(I_1)\le t_0\le t(I_1+1)$. Since a.a.s.\ $L(\hG_{I_1})=L_{t(I_1)}=\Theta(n\deltac)$, by~\eqn{theta-diff}, we have $\theta_{t(I_1)}=\theta_{t_0}+O(L(\hG_{I_1})/n)=\Theta(\sqrt{\deltac})$ as $\deltac=o(\sqrt{\deltac})$.

Finally, $t(I_1)\le t_0=Y_3 n\sqrt{\vs\deltac}$ by our definition of $I_1$. We also have $t(I_1)\ge t_0-L(\hG_{I_1})=\Theta(t_0)$ since $L(\hG_{I_1})=L_{t(I_1)}=O(n \xi)=o(t_0)$. This shows that $t(I_1)=\Theta(n\sqrt{\deltac})$ as required and this completes our proof of the lemma by letting $i=I_1$ in the statement of the lemma.
 \qed\smallskip

\remove{************
********************
********************
By~\eqn{Ldiff} and the definition of $\theta_t$, if $L_t>0$,
\[
\ex(L_{t+1}\mid G_t)=L_t+\theta_t-(\theta_t+r)\frac{L_t}{B_t}+O(n^{-1}).
\]
Define $\htau'$ to be the minimum integer that $L_t\ge Y_0 n \sqrt{\vs \deltac}$ for some small constant $Y_0>0$ to be specified later; and define $\htau:=\min\{\htau'-1,t(T_1),\tau_2(\sigma)\}$.
 By Lemma~\ref{lem:theta0}, there is a constant $Y_1>0$ such that a.a.s.\ $\theta_0\ge -Y_1 \sqrt{\vs \deltac}$.
 By Corollaries~\ref{cor:zeta1} and~\ref{cor:hit}, for all $0\le t\le t(T_1)$, $\theta_t\ge -2Y_1 \sqrt{\vs \deltac}$. Hence, $(1-L_t/B_t)\theta_t\ge -2Y_1 \sqrt{\vs \deltac}$. By the definition of $\htau$, for all $t\le\htau$, $L_t/B_t\le (Y_0/\sigma)\sqrt{\vs \deltac} $ since $B_t\ge \sigma n$. So
\[
\ex(L_{t+1}\mid G_t)\ge L_t-\Big(2Y_1 \sqrt{\vs} +(r Y_0/\sigma)\sqrt{\vs} \Big)\sqrt{\deltac}+O(n^{-1})\ge L_t - 3Y_1 \sqrt{\vs \deltac},
\]
by choosing $Y_0<\sigma Y_1/r$.
Applying Lemma~\ref{l:azuma} to $(-L_t)_{t\ge 0}$ (with $a_n=3Y_1\sqrt{\vs \deltac}$, $c_n=r$ and $j=ta_n$) we have that for every $0\le t_1<t_2\le \htau$ with $t_2-t_1\ge \log^2 n/\sqrt{\deltac}$,
\be
\pr(L_{t_2}\le L_{t_1}-t\cdot 6Y_1 \sqrt{\vs \deltac}) \le \exp\Big(-\Omega\big(\sqrt{\vs \deltac}(t_2-t_1)\big)\Big)<n^{-2}. \lab{Lbound}
\ee
By Lemma~\ref{lem:L0}, a.a.s.\ $D_0=L_0=\Theta(n \cdot\deltac)$; this implies that a.a.s.\ $D_0/r\le t(1)\le D_0$, since each step of SLOW-STRIP removes at most $r$ points in the vertices in $S_0$. Condition on that. By~\eqn{Lbound}, with probability at least $1-n^{-2}$,
$D_1=L_{t(1)}\ge L_0-6Y_1\sqrt{\vs \deltac}\cdot  t(1)\ge D_0(1-6Y_1 \sqrt{\vs} n^{-\d/2})$. Inductively, for each $i\ge 1$, with probability at least $1-n^{-2}$, $D_{i+1}\ge D_i(1-6Y_1 \sqrt{\vs \deltac})$, provided $D_i\ge \log^2 n/\sqrt{\deltac}$. Note that $\htau\le t(T_1)$. Since $T_1=1/\sqrt{\deltac}$, we have that for some constant $Y_2>0$, $D_0 (1-Y_2\sqrt{\vs \deltac})^{i} =\Omega(D_0)\ge n^{\d/2}\log^2 n$ for all $0\le i\le T_1$. Hence, inductively
we have a.a.s.\  $D_{i+1}\ge D_i(1-Y_2 \sqrt{\vs \deltac})$ for all $i\ge 0$ such that $t(i)\le \htau$.

Assume $t(T_1)\le \htau$. Then, by the above discussion and noting that $D_0=\Omega(n\cdot\deltac)$, we have
\be
\sum_{i=0}^{T_1} D_i \ge D_0 C_2 /\sqrt{\vs\deltac} \ge C_3 n\sqrt{\deltac}/\sqrt{\vs}, \lab{sumD}
\ee
 for some constant $C_2,C_3>0$.
 Since in each step of SLOW-STRIP, at most $r$ points contained in light vertices are removed, we must have
 \[
 t(T_1)\ge (C_3/r) n\sqrt{\deltac}/\sqrt{\vs}\ge (C_3/r) n\sqrt{\deltac'}/\sqrt{\vs}
  \]
  (since $\deltac\ge \deltac'$ as long as $\vs<1$). By Corollary~\ref{cor2:hit}, by  Choosing $\vs$ sufficiently small (so that $C_3/r\sqrt{\vs}$ is sufficiently large), a.a.s.\  $\theta_{t(T_1)}\ge (Y_3C_3/r)\sqrt{\deltac}/\sqrt{\vs}$ for some constant $Y_3>0$ and $L_{t(T_1)}=D_{T_1}\ge D_0 \exp(-Y_2\sqrt{\vs})=\Omega(n\deltac)$ which implies our lemma.

Now we may assume that $t(T_1)>\htau$. Let $i_0$ be the maximum integer such that $t(i_0)\le \htau$. Thus, $i_0<T_1$. Then, by the definition of $\htau$, $L_{t'}\ge Y_0n\sqrt{\vs \deltac}$ for some $t(i_0)\le t'\le t(i_0+1)$. Then, for all $t>t'$, either $L_t\ge Y_0n\sqrt{\vs \deltac}$ or
\[
\ex(L_{t+1}\mid G_t)\ge L_t-3Y_1\sqrt{\vs \deltac}.
\]
Assume $\htau=o(n)$, then it is easy to see that $L_t\ge Y_0n\sqrt{\vs \deltac}/2$ for all $t'\le t\le \htau$, since a.a.s.\ it takes $\Omega(n)$ steps for $L_t$ to drop below $Y_0n\sqrt{\vs \deltac}/2$ from $Y_0n\sqrt{\vs \deltac}$.
In this case, we will certainly have~\eqn{sumD} as $Y_0n\sqrt{\vs \deltac}/2$ is much greater than $D_0$. Therefore, it follows also that a.a.s.\ $\theta_{t(T_1)}=\Omega(\sqrt{\deltac})$ and $L_{t(T_1)}=\Omega(n\deltac)$.

Assume $\htau=\Omega(n)$; then $t_1=\Omega(n)$ and the lemma is implied by Corollaries~\ref{cor2:zeta} and~\ref{cor2:hit}. It is trivial to verify $L_t=\Omega(n)$ as well and so our claim follows.
******************
********************
}

Let $I_1$ be the integer specified in Lemma~\ref{lem:D}; so a.a.s.\ $\theta_{t(I_1)}=\Theta(\sqrt{\deltac})$, $L_{t(I_1)}=\Theta(n\deltac)$ and $I_1=O(1/\sqrt{\deltac})$.

%



\begin{lemma}\lab{lem:D2} Suppose $\eps>0$ is sufficiently small. There are constants $Y_1,Y_2>0$ such that
a.a.s.\ for all $t(I_1)\le t\le \eps n$, $Y_1 t^2/n\le L_t\le Y_2 t^2/n$.
\end{lemma}

\proof Let $t_0=t(I_1)$.  We will first prove the upper bound.  It certainly  holds a.a.s.\ for $t=t_0$, as a.a.s.\ $t_0=\Theta(n\sqrt{\deltac})$, $\theta_{t_0}=\Theta(\sqrt{\deltac})$ and $L_{t_0}=\Theta(n\deltac)$ by Lemma~\ref{lem:D}. By Lemma~\ref{lem:zeta}(c), Corollary~\ref{cor:hit} and the fact that a.a.s.\ $\theta_{t_0}=\Theta(\sqrt{\deltac})$, a.a.s.\ $C_1 t/n\le \theta_t\le C_2 t/n$ for some constants $C_1,C_2>0$ for all $t_0\le t\le \eps n$. By~\eqn{Ldiff2}, a.a.s.\ for every $t_0\le t\le \eps n$ (noting that $\theta_t>0$ in this range),
 \[
\ex(L_{t+1}-L_t\mid \calf_t)\le \theta_t+O(n^{-1})\le \frac{2C_2t}{n}.
\]
 Then by Lemma~\ref{l:azuma}, for all $t_0+\log^2 n\le t\le \eps n$, a.a.s.\ $L_{t}\le L_{t_0}+4C_2 t(t-t_0)/n \le Y_2 t^2/n$ by choosing sufficiently large $Y_2$. For all $t$ between $t_0$ and $t_0+\log^2 n$, $L_t=L_{t_0}+O(\log^2 n)$ and thus the upper bound trivially holds for sufficiently large $Y_2$, since $Y_2t^2/n\ge Y_2t_0^2/n=\omega(\log^2 n)$.

 Next, we prove the lower bound.
Clearly, it  holds a.a.s.\ for $t_0$ for some constant $Y_1=Y$. Up to step $\eps n$, at most $(r-1)\eps n$ heavy vertices can be removed, and the number of heavy vertices in $G_{0}$ is $(\a+o(1))n$, as we have discussed before. Hence, the total degree, $B_t$, of $G_t$, for any $t\le \eps n$, is at least $k(\alpha - r\eps) n$.
 Now, by the upper bound we have just shown, a.a.s.\ for all $t_0\le t\le \eps n$, $L_t\le Y_2 t^2/n$ and hence for all $t$ in this range,
 \[
 \frac{rL_t}{B_t}\le \frac{rY_2 t^2/n}{k(\a-r\eps)n}\le (rY_2\eps/\a)\frac{t}{n}<(C_1/4)\frac{t}{n}<\frac{1}{4},
   \]
   by choosing $\eps>0$ sufficiently small.
   Then,
by~\eqn{Ldiff2}, a.a.s.\ for every $t_0\le t\le \eps n$,
\be
\ex(L_{t+1}-L_t\mid \calf_t)\ge \frac{3}{4}\theta_t-(C_1/4)\frac{t}{n}+O(n^{-1})\ge (C_1/3)\frac{t}{n},\lab{L}
\ee
as $\theta_t\ge C_1t/n$ for all $t$ in this range.

We split the range $t_0\le t\le \eps n$ into intervals, each with length $t_0$ (hence the $j$-th interval is from $(j-1)t_0$ to $jt_0$) and the last interval is simply the remainder.
Similar to our analysis in Lemma~\ref{lem:D}, for the $j$-th interval,
the probability that $L_t<L_{(j-1)t_0}+(C_1/4) ((j-1)t_0/n)\cdot(t-(j-1)t_0)$ is at most $n^{-2}$ for all $(j-1)t_0+\log^2 n\le t\le jt_0$, and we always have $L_t=L_{(j-1)t_0}+O(t-(j-1)t_0)$ for all $(j-1)t_0\le t\le (j-1)t_0+\log^2 n$.

Since $t_0=\Theta(n\sqrt{\deltac})$, the total number of intervals is $O(1/\sqrt{\deltac})$. So, a.a.s.\ for every interval $j$,
\bean
L_{(j-1)t_0+d}&\ge& L_{(j-1)t_0}+\frac{C_1 (j-1)}{4n}t_0 d,\ \mbox{if}\ d\ge \log^2 n\\
L_{(j-1)t_0+d}&=&L_{(j-1)t_0}+O(d),\ \mbox{if}\ d<\log^2 n.
\eean
It is easy to verify that by choosing $Y'= \min\{Y, C_1/12\}$,
a.a.s.\ $L_t\ge Y' t^2/n$ for all $t_0\le t\le \eps n$ except for the first $\log^2 n$ numbers in each interval. But then the inequality must hold by choosing $Y_1=Y'/2$ since $t_0^2/n$ is $\omega(\log^2 n)$. This completes the proof for the lower bound. \qed\smallskip

Suppose $\eps>0$ is chosen to satisfy Lemma~\ref{lem:D2}. Define
\be
I_2=\max\{i:\ t(I)\le \eps n \}.\lab{I2}
 \ee 
 In the following lemma, we bound $I_2$.

\begin{lemma}\lab{lem:I2}
A.a.s.\ $I_2=O(\deltac^{-1/2}\log (1/\deltac))$.
\end{lemma}
\proof We have shown in~\eqn{L} that a.a.s.\ for all $t_0\le t\le t(I_2)$, where $t_0=t(I_1)$,
\[
\ex(L_{t+1}-L_t\mid \calf_t)\ge C_1t_0/n\ge C_2 \sqrt{\deltac},
\]
for some constant $C_1,C_2>0$. Recalling~\eqn{t(i)} and by Lemma~\ref{l:azuma},
this immediately gives that a.a.s.\
\[
L_{t(i+1)}\ge (1+(C_2/2)\sqrt{\deltac}) L_{t(i)}, \ \mbox{for all}\ I_1\le i\le I_2-1.
\]
Since $L_{t(I_1)}=\Theta(n\deltac)$ by Lemma~\ref{lem:D} and $L_{t(I_2)}=O(n)$, it follows immediately that $I_2-I_1=O(\deltac^{-1/2}\log (1/\deltac))$. The lemma follows as $I_1=O(1/\sqrt{\deltac})$ by Lemma~\ref{lem:D}. \qed\smallskip


Recall from~\eqn{tau2} that $\tau_2(\sigma)$ is the last step after which $N_t\ge \sigma n$ (c.f.\ $\tau_1(\sigma)$ for the sequence $(\M_t)$, defined in Section~\ref{sec:M}).

\begin{lemma}\lab{lem:tau3} There is a constant $\eps_0>0$ such that
a.a.s.\ $L_t\ge \eps_0 n$ for all $t(I_2)\le t\le \tau_2(\sigma)$.
\end{lemma}

\proof We first prove that a.a.s.\ there is $t_1\le t(I_2)$ such that $L_{t_1}=\Theta(n)$. By Lemma~\eqn{lem:D2}, for all $t(I_1)\le t\le \eps n$, 
\be
Y_1 t^2/n\le L_t\le Y_2t^2/n,\lab{Ltbounds}
\ee
for some constants $Y_1,Y_2>0$. Let $\eps_0=\eps/(1+Y_2)$. Then,
$\eps_0<\eps$ and so $L_{\eps_0 n} \le Y_2 \eps_0^2 n$. Let $i_1=\max\{i:\ t(i)\le \eps_0 n\}$. Then, $t(i_1)\le \eps_0 n< t(i_1+1)$. Moreover $t(i_1+1)\le \eps_0 n + L_{\eps_0 n} \le \eps_0 n+ Y_2\eps_0^2 n\le \eps_0(1+Y_2) n =\eps n$. Let $t_1=\eps_0 n$. Then $t_1\le t(i_1+1)\le I_2$ by~\eqn{I2}; moreover, $L_{t_1}$ satisfies~\eqn{Ltbounds}. So, $L_{t_1}=\Theta(n)$.

By Corollaries~\ref{cor2:zeta} and~\ref{cor2:hit}, there is a constant $\eps_1>0$ such that a.a.s.\
$\theta_t\ge \eps_1$ for all $t_1\le t\le \tau_2(\sigma)$. We may assume $\eps_1<1$.
    For all $t\le\tau_2(\sigma)$ we have $B_t\ge k\sigma n$.
    
     Let $\eta=\min\{(Y_1/2)\eps_0^2, (\eps_1k\sigma/12r)\}$. Next, we prove a.a.s.\ for all $t_1\le t\le \tau_2(\sigma)$, $L_t\ge \eta n$. For $t=t_1$, this is true since a.a.s.\ $L_{t_1}\ge 2\eta n$ by~\eqn{Ltbounds}. Let $A_t$ be the event that $L_t\ge 2\eta n -kr$ and $L_{t'}<2\eta n$ for all $t\le t'\le t+\eta n/kr$. Since $L_t$ changes by $kr$ in each step, if $L_{t'}\le \eta n$ for some $t_1\le t'\le \tau_2(\sigma)$, then $A_t$ must occur for some $t_1\le t\le t'-\eta n/kr$. Next, we bound the probability of $A_t$. We may assume $2\eta n -kr\le L_t<2\eta n$ since otherwise $A_t$ does not hold.
     Since $L_t$ changes by $kr$ in each step, we have
     $L_{t'}\le 3\eta n\le (\eps_1 k\sigma/4r)n$ for all $t\le t'\le t+\eta n/kr$. Then, $rL_t/B_t \le \eps_1/4$. Then by~\eqn{Ldiff}, for all $t\le t'\le t+\eta n/kr$,
\[
\ex(L_{t+1}-L_t\mid \calf_t)\ge \eps_1/2.
\]
     By Lemma~\ref{l:azuma},
     \[
     \pr(A_t)\le \pr(L_{t+\eta n/kr}<2\eta n \mid L_t\ge 2\eta n-kr)=o(n^{-1}).
     \]
     By the union bound, the probability that $A_t$ occurs for some $t$ is $o(1)$ and thus, a.a.s.\ $L_t\ge \eta n$ for all $t_1\le t\le \tau_2(\sigma)$. \qed\smallskip

Now we complete the proof of Theorem~\ref{thm:Stripsub}(a).
By the definition of $I_{\sigma}$ in~\eqn{Isigma},  $\hG_{I_{\sigma}+1}$ contains less than $\sigma n$ vertices. By Lemma~\ref{lem:small}, every component in $\hG_{I_{\sigma}+2}$ has $O(\log n)$ vertices. It was proved in~\cite{g2} that each of such components has stripping number $O(\log\log n)$ (the basic idea there is that, a.a.s.\ every subgraph of $AP_r(n,cn)$ (for any $c=\Theta(1)$) with size $O(\log n)$ is so sparse that, when the parallel stripping process is applied to it, a positive proportion of vertices are stripped off in each round, and thus a total of $O(\log\log n)$ iterations is sufficient). Hence, a.a.s.\ $s(G_0)\le I_{\sigma}+2+O(\log\log n)$. By Lemma~\ref{lem:tau3}, a.a.s.\  $|S_i|=\Omega(L_{t(i)})=\Omega(n)$ for every $I_2\le i\le I_{\sigma}$, and so a.a.s.\ $I_{\sigma}-I_2=O(1)$. By  Lemmas~\ref{lem:I2}, a.a.s.\ $I_2=O(\deltac^{-1/2}\log(1/\deltac))$. Recall that $G_0=\hH_{\tau'(B)}$ by definition~\eqn{G0}.  It follows then that a.a.s.\
$s(\hH_{\tau'(B)})=s(G_0)=O(\deltac^{-1/2}\log(1/\deltac)+\log\log n)$.
Then Theorem~\ref{thm:Stripsub}(a) follows by~\eqn{couple},~\eqn{tauprimeBound} and Corollary~\ref{ccon}.\qed

\subsection{Inside the critical window: proof of Theorem~\ref{thm:Stripsub}(b)}
\lab{sec:b}

Now we fix a constant $\eps>0$, and consider $c$ such that $|c-c_{r,k}|\le n^{-1/2+\eps}$. Again, we may assume that $\eps$ is sufficiently small.
Let $c'=c_{r,k}+n^{-1/2+2\eps}$. By Theorem~\ref{thm:Stripsuper}(a), a.a.s.\ $\tau(H')=O(n^{1/4-\eps}\log n)$. Since $c'$ satisfies all conditions in~\eqn{condc}, all lemmas and corollaries in Sections~\ref{sec:M} and~\ref{sec:theta} hold. However, we note here that $\deltac$ and $\deltac'$ are no longer of the same order, and so we cannot replace $\deltac'$ by $\deltac$ in any asymptotic expression. By Lemma~\ref{lem:theta0} we have
$\zeta_0=\zeta+\Theta(n^{-1/4+\eps})$ and $\theta_0=\theta-\Theta(n^{-1/4+\eps})$.

Let $G_0$ be as defined in~\eqn{G0}. Recall that $G_0,G_1,\ldots$ is the sequence produced by SLOW-STRIP. Let $\htau$ denote the step when SLOW-STRIP terminates. Since $c$ is inside the critical window $c_{r,k}+O(n^{-1/2+\eps})$, whether $G_{\tau}$ is empty or not is not a.a.s.\ certain.

Let $K>0$ be a constant to be determined later; define $t_1=K n^{3/4+\eps}$ (i.e.\ $t_1=Kn\sqrt{\deltac'}$).
\begin{lemma} \lab{lem:bigL}
Assume $K>0$ is sufficiently large and $\eps$ is sufficiently small. Then, a.a.s.\ either $\htau\le t_1/2$; or there is $t\le t_1$ that $L_t=\Omega(n^{1/2+2\eps})$.
\end{lemma}

\proof  

Since $t_1=Kn\sqrt{\deltac'}$,
 by Corollary~\ref{cor2:hit} , provided $K>0$ is sufficiently large, we have a.a.s.\ either $\htau\le t_1$ or $C_1 t/n\le \theta_t \le C_2 t/n$ for some constants $C_1,C_2>0$ and for all $t_1/2\le t\le \tau_2(\sigma)$.

Let $Y_3>0$ be a constant to be specified later.
Define $\tau_3$ to be the minimum integer $t\ge t_1/2$ such that $L_{t+1}\ge Y_3 n^{3/4+\eps}$. If no such integer $t$ exists, then define $\tau_3=n$. Define $T_1=\min\{t_1,\tau_3,\tau_2(\sigma)\}$. Then, for all $t_1/2\le t\le T_1$, we have $L_t<Y_3 n^{3/4+\eps}$ and $B_t\ge k\sigma n$. Therefore, $L_t/B_t\le (Y_3/k\sigma) n^{-1/4+\eps}$. Hence $1-L_t/B_t\ge 2/3$ and $rL_t/B_t\le (rY_3/k\sigma)n^{-1/4+\eps}$. By~\eqn{Ldiff2}, for all $t_1/2\le t\le T_1$,
\[
\ex(L_{t+1}-L_t\mid \calf_t) \ge(2C_1/3) \frac{t}{n} - (rY_3/k\sigma)n^{-1/4+\eps}+O(n^{-1}) \ge (C_1K/4) n^{-1/4+\eps},
\]
by choosing $Y_3<k\sigma C_1 K/12r$, since $t\ge t_1/2=(K/2)n^{3/4+\eps}$.

We may assume that $\tau_3\ge t_1$ since otherwise there is $t_1/2\le t\le t_1$ such that $L_t\ge Y_3n^{3/4+\eps}$ and so claim of the lemma is verified. By the definition of $\tau_2(\sigma)$, we have $\tau(\sigma)>t_1$ always, as $\tau_2(\sigma)$ is linear in $n$. Hence, we may assume that $T_1=t_1$.

Applying Lemma~\ref{l:azuma} to $(L_t)_{t=t_1/2}^{t_1}$ (with $a_n=-(C_1K/4)n^{-1/4+\eps}$, $c_n=kr$ and $j=-ta_n/2$), we have a.a.s.\
$L_{t_1}\ge (C_1K/8)n^{-1/4+\eps}(t_1/2) = (C_1 K^2/16) n^{1/2+2\eps}$.
This means that a.a.s.\ either $\htau\le t_1/2$, or there is $t_1/2\le t\le t_1$ such that $L_t\ge Y_3n^{3/4+\eps}$ (i.e.\ $\tau_3\ge t_1$), or $L_{t_1}=\Omega(n^{1/2+2\eps})$. In each case, the claim of the lemma is verified (as $Y_3n^{3/4+\eps}>n^{1/2+2\eps}$ by taking $\eps<1/4$). \qed \smallskip

We may assume that there is a $t_1/2\le t_0\le t_1$ such that $L_{t_0}\ge n^{1/2+2\eps}$, since otherwise $\htau\le t_1$ by Lemma~\ref{lem:bigL} and so $s(G_0)\le t_1=O(n^{3/4+\eps})$. Next, we prove that the parallel stripping process does not take long from $G_{t_0}$.

\begin{lemma} \lab{lem:fast}
Assume $L_{t_0}\ge n^{1/2+2\eps}$ for some $t_1/2\le t_0\le t_1$. Then, a.a.s.\ $s(G_{t_0})=O(n^{1/2})$.
\end{lemma}

\proof Since $t_0\ge (K/2)n^{3/4+\eps}=(K/2)n\sqrt{\deltac'}$, by Corollary~\ref{cor2:hit}, provided that $K>0$ is sufficiently large, we have $\theta_{t}\ge 2C t/n \ge CKn^{-1/4+\eps}$ for all $t_0\le t\le \tau_2(\sigma)$, for some constant $C>0$. We will prove that a.a.s.\ there is no $t_0\le t\le \tau_2(\sigma)$ such that $L_t<n^{1/2}$. Let $A_t$ be the event that $L_t\ge 2n^{1/2}-kr$, and $L_j<2n^{1/2}$ for all $t\le j\le t+n^{1/2}/kr$. Since $|L_{t+1}-L_t|\le kr$ always, $L_{t_0}\ge n^{1/2+2\eps}$, if $i$ is the first step after $t_0$ such that $L_i<n^{1/2}$, then, there must be some $i'\le i-n^{1/2}/kr$ such that $A_{i'}$ occurs. Hence, it is sufficient to prove that a.a.s.\ there is no $t_0-n^{1/2}/kr\le i\le \tau_2(\sigma)-n^{1/2}/kr$ for which $A_i$ occurs. 

Since $L_{t_0}\ge n^{1/2+2\eps}$, $A_i$ cannot occur for any $t_0-n^{1/2}/kr\le i\le t_0$ by definition of $A_i$. Next,
for each $t_0<i\le \tau_2(\sigma)-n^{1/2}/kr$, we bound the probability of $A_i$. We may assume that $2n^{1/2}-kr\le L_i<2n^{1/2}$ since otherwise, $A_i$ does not hold by definition. Let $T$ be the minimum integer such that $T>i$ and $L_{T+1}\ge 2n^{1/2}$. Define $T=n$ if no such integer exists. Now, for all $i\le t\le T$, $L_t<2n^{1/2}$ and so
 $L_t/B_t=o(n^{-1/4+\eps})$. By~\eqn{Ldiff2} and the fact that $\theta_t\ge CKn^{-1/4+\eps}$ for all $t_0\le i\le t\le \tau_2(\sigma)$,
\be
\ex (L_{t+1}-L_t\mid \calf_t)\ge (1/2)\theta_t+o(n^{-1/4+\eps})\ge (CK/2)n^{-1/4+\eps}.\lab{diff}
\ee
Now,
\[
\pr(A_i)\le \pr(T\ge i+n^{1/2}/kr) = \pr(L_{t}<2n^{1/2},\ \forall i\le t\le i+n^{1/2}/kr).
\]
However, since $L_i\ge n^{1/2}$ by assumption, by applying Lemma~\ref{l:azuma} to~\eqn{diff}, with probability $1-o(n^{-1})$, we must have $L_{i+n^{1/2}/kr}\ge L_i +\Omega(n^{1/4+\eps})>2n^{1/2}$. Hence, $\pr(A_i)=o(n^{-1})$ for every $i$. This confirms that a.a.s.\ $L_t\ge n^{1/2}$ for all $t_0\le t\le \tau_2(\sigma)$. Therefore, a.a.s.\ the total degree of the configuration decreases by at least $n^{1/2}$ in each iteration of the parallel stripping process applied to $G_{t_0}$, until there are at most $\sigma n$ vertices remaining.
Now $s(G_{t_0})=O(n/n^{1/2}+\log n)=O(n^{1/2})$ by Lemma~\ref{lem:small}.
 \qed \smallskip

Now we complete the proof of Theorem~\ref{thm:Stripsub}(b). If SLOW-STRIP terminates by step $t_1$, then
$s(G_0)=O(n^{3/4+\eps})$. Hence,
$s(H)\le \tau'(B)+1+s(G_0)=O(n^{3/4+\eps})$ by~\eqn{couple},~\eqn{G0} and~\eqn{tauprimeBound}. If not, by Lemma~\ref{lem:bigL} we may assume that there is a step $t_0\le t_1$ such that $L_{t_0}\ge n^{1/2+2\eps}$. By Lemma~\ref{lem:fast}, a.a.s.\ $s(G_{t_0})=O(n^{1/2})$. Hence, a.a.s.\ $s(H)\le \tau'(B)+1+t_1+s(G_{t_0})=O(n^{3/4+\eps})$. This implies part (b) of Theorem~\ref{thm:Stripsub}. \qed

\section{Bounding the maximum depth}

Let $H_0\in AP_r(n,cn)$ where $c=c_{r,k}-\xi_n$, and let $H_0,H_1,\ldots,$ denote the parallel stripping process. For simplicity, we write $\xi_n$ as $n^{\delta}$ ($\delta$ may depend on $n$). Our goal is to show that the maximum depth of all the non-$k$-core vertices of $AP_r(n,cn)$ is a.a.s.\ $n^{\Theta(\delta)}$.
 
Recall that $S_i$ is the set of light vertices in $H_i$. Let $\imax$ denote the last iteration of the parallel stripping process; i.e.\ $\imax=s(H_0)$.
 Let $v$ be a vertex in $S_{\imax-1}$. Then every stripping sequence ending with $v$ must contain at least one vertex in each $S_i$, $1\le i\le \imax-1$. Hence, the depth of $v$ is at least $I_{\max}-1$, which is a.a.s.\ $\Omega(n^{\d/2})$ by Theorem~\ref{thm:Stripsuper}(b,c). This verifies the lower bound in Theorem~\ref{thm:depthsub}. The main challenge of this section is to prove the upper bound. 

Let $\Psi=v_1v_2\cdots$ be a stripping sequence that contains all non-$k$-core vertices of a hypergraph $H$. Let $\cald(\Psi)$ be a digraph constructed as follows.
The set of vertices in $\cald(\Psi)$ is $V(H)$. At the moment $v_i$ is removed from the hypergraph $H$, consider each hyperedge $x$ that is incident with $v_i$ before the removal of $v_i$, add a directed edge $u\rightarrow v$ to $\cald(\Psi)$, for each of the other $\dd-1$ vertices $u\in x$.  For any vertex $v\in H$, we define $R_{\Psi}^+(v)$ to be the set of vertices reachable from $v$ in $\cald(\Psi)$.
Clearly,  $R_{\Psi}^+(v)$ forms a stripping sequence ending with $v$ and so $|R^+(v)|$ is an upper bound of the depth of $v$.
We will prove the following stronger statement than the upper bound in Theorem~\ref{thm:depthsub}. 
\begin{theorem}\lab{depth2} 
Let $r,k\geq2, (r,k)\neq (2,2)$ be fixed and assume $c=c_{r,k}-n^{-\d}$ where $n^{
\d}\ge \log^7 n$ (here $\d$ can depend on $n$). Then, a.a.s.\ there is a stripping sequence $\Psi$ of $\calh_r(n,cn)$ containing all non-$k$-core vertices such that 
\begin{equation}\lab{emt2}
|R_{\Psi}^+(v)|= n^{O(\d)}\quad \mbox{for all $v\in\Psi$}.
\end{equation}
\end{theorem}

It was shown in~\cite{gm} that~\eqn{emt2} a.a.s.\ holds if $c=c_{r,k}+n^{-\d}$. The analysis for the subcritical case is analogous to the supercritical case. Thus, we summarise the approach in~\cite{gm} and sketch a proof for Theorem~\ref{depth2} by pointing out how the arguments in~\cite{gm} should be adapted. 

Let $H\in AP_r(n,cn)$ and run SLOW-STRIP on $H$. This produces a stripping sequence $\Psi$. Let $\cald=\cald(\Psi)$ be the digraph associated with $\Psi$. This defines $R^+_{\Psi}(v)$ for any non-$k$-core vertex $v$. For simplicity, we drop the subscript $\Psi$ from the notation.

\begin{definition}\lab{def:Rj}
For each $0\le j\le \imax-1$, define $\cals_j$ to be a graph with vertex set $S_j$ and edge set $\{f\cap S_j: f\in E(H_j)\}$.

Take an arbitrary vertex $v\in S_i$; set
$R'_i=R'_i(v):=\{v\}$ and for each $j=i$ to $0$:
\begin{enumerate}
\item[(a)] We set $R_j=R_j(v)$ to be the union of the vertex sets of all components of $\cals_j$ that contain vertices of $R'_j$.
\item[(b)] We set $R'_{j-1}$ to be the set of all vertices $v\in S_{j-1}$ that are adjacent to $\cup_{\ell=i}^{j} R_{\ell}$.
\end{enumerate}
Define $R(v)=\cup_{j=0}^i R_{j}$.
\end{definition}

Note that $\cup_{j\le i}R_j$ contains all vertices reachable from $v$ and therefore $|R^+(v)|\le \cup_{j\le i}|R_j|$. The definition of $R_j(v)$ makes it easier (compared with $|R^+(v)|$) to bound $|R_i|$, given $|R_j|$ ($j>i$).

\subsection{Proof outline for the supercritical case~\cite[Section 5]{gm}}

Starting with a vertex $v\in S_i$, we explore the hyperedges from $v$ to other vertices in $S_i$ and $S_{i-1}$, and then onwards; recursively, we will bound $|R_j|$ for each $j\le i$. We need randomness to allow such analysis; but we can expose $v\in S_i$ only after we expose all vertices in $S_j$, $j<i$ and their incident hyperedges, and that has killed all the randomness. A solution to this obstacle is given in~\cite{gm}. Basically, we will expose all vertices in each $S_i$, $0\le i\le \imax-1$, as well as some degree information. For instance, we need information on the number of total neighbours $S_i$ has in $S_{i-1}$, etc. This procedure is called EXPOSURE. Then we will generate uniformly a random configuration that agrees with the exposed information. This procedure is called EDGE-SELECTION (similar to the configuration model). See~\cite[Section 5.1]{gm} for details. It was proved (\cite[Lemma 46]{gm}) that EDGE-SELECTION generates the random configuration with the correct distribution, conditional on the information exposed in EXPOSURE.

Now, conditional on the set of parameters exposed in EXPOSURE, we can recursively bound $|R_i|$, given $|R_j|$, $j> i$, as EDGE-SELECTION defines a probability space that is easy for such analysis. Analysing SLOW-STRIP allows us to prove a.a.s.\ properties of the set of parameters exposed in EXPOSURE. This gives~\cite[Lemma 49(a--i)]{gm}, where the first three of these properties have been stated in Lemma~\ref{lem:Si}. We will state the corresponding properties for the subcritical case in Section~\ref{sec:subdepth}.

We first introduce two lemmas that allow
us to focus our analysis on $R_j(v)$ for $j$ such that the number of vertices in $H_j$ is close to $\a n$.
The first lemma is proved in~\cite[Lemma 15]{gm}, with $c=c_{r,k}+n^{-\d}$ for some $0<\d<1/2$; but exactly the same proof gives the following statement with all $c=c_{r,k}+o(1)$.
\begin{lemma}\lab{lem:Hb}
Assume $c=c_{r,k}+o(1)$. Given any constant $\eps>0$, there is a constant $B=B(r,k,\eps)>0$ such that after $B$ rounds of the parallel stripping process are applied to $AP_r(n,cn)$, the number of vertices remaining in $H_B$ is $(\a+O(\eps))n$.
\end{lemma}

Given a set of vertices $A$, let $N^s(A)$ denote the set of vertices with distance at most $s$ from $A$. The following lemma is from~\cite[Lemma 34]{amxor}.
\begin{lemma}\lab{lem:neighbours}
Assume $s,c>0$ are $O(1)$. A.a.s.\ for every subset of vertices $A$ in $AP_r(n,cn)$ such that $A$ induces a connected subgraph, $|N^s(A)|=O(|A|+\log n)$.
\end{lemma}

In the supercritical case, $H_0,H_1,\ldots, H_{\imax}$ is the parallel stripping process with $H_0\in AP_r(n,cn)$, where $c\ge c_{r,k}+n^{-\d}$. Fix an $\eps>0$. By Lemma~\ref{lem:Hb}, there is a constant $B>0$ such that the number of vertices in $H_B$ is at most $(\a+\eps)n$. Assume we have successfully bounded $\sum_{B\le j\le i}|R_j(v)|$ for some $v\in S_i$; then, by Lemma~\ref{lem:neighbours}, a.a.s.\
\be
|R^+(v)|\le \sum_{0\le j\le i}|R_j| =O\left(\log n+ \sum_{B\le j\le i}|R_j|\right).\lab{R(v)}
\ee
Therefore, in order to bound $|R^+(v)|$, it is sufficient to bound $|R_j(v)|$ for every $B\le j\le i$.

Let ${\mathcal E}_i$ denote the set of $r$-tuples incident with $S_i$ in $H_i$; i.e.\ it is the set of hyperedges to be deleted in the $(i+1)$-th iteration. For each $v\in S_{i+1}$, define $d^-(v)$ to be the number points in $v$ that are contained in an $r$-tuple in ${\mathcal E}_i$. In the case $r=2$, $d^-(v)$ is simply the number of neighbours of $v$ in $S_i$ in $H_i$. Naturally, we have $d^-(v)\ge 1$ for every $v\in S_{i+1}$, since otherwise, $v$ would not be deleted in the $(i+1)$-th iteration. Define $D^-(X)=\sum_{v\in X} d^-(X)$. Then, $|R_j(v)|\le D^-(R_j(v))$.
A major part of the work in~\cite[Section 5]{gm} is to recursively bound $D^-(R_j(v))$,
using the random configuration generated by EDGE-SELECTION, and a set of a.a.s.\ properties in~\cite[Lemma 49]{gm}, as follows.

\begin{lemma}\lab{lrec} If all properties in~\cite[Lemma 49(a--i)]{gm} hold then there are constants $B=B(r,k),Z=Z(r,k)>0$ such that for all $j\geq B$: { If $|S_j|\geq n^{\d}\log^2 n$ then: }
\be
D^-(R_j(v))\leq D^-(R_{j+1}(v))+Z\frac{|S_j|}{n}\sum_{\ell=i}^{j+1}D^-(R_{\ell}(v)) +\log^{14}n.\lab{rec}
\ee
\end{lemma}

Solving this recurrence, together with~\eqn{R(v)} produces $|R^+(v)|=n^{O(\d)}$ for all non-$k$-core vertices $v$.

\subsection{Sketch of the proof of Theorem~\ref{depth2}}
\lab{sec:subdepth}

  It only remains to prove the upper bound. We may assume that $c=c_{r,k}-n^{-\d}$ for some $\delta<1/2$ since for the case that $\d\ge 1/2$, the upper bound in the theorem is trivial, as discussed in the remark below Theorem~\ref{depth2}.  The approach is the same as in the supercritical case. The same set of a.a.s.\ properties (c.f.~\cite[Lemma 49(a--i)]{gm} for the supercritical case) are exposed in EXPOSURE. Properties in~\cite[Lemma 49(d--i)]{gm} carry to the subcritical case, as the proof of these properties only requires $c=\Theta(1)$. The proof of Lemma~\ref{lrec} depends mainly on these properties, as well as the restriction that $|S_i|/n$ is sufficiently small. Hence,
we will have the same recursive bound as in Lemma~\ref{lrec} for most iterations in the subcritical case. However, $|S_i|$ behaves in a different manner in the subcritical case, which results in a different argument in solving the recurrence~\eqn{rec}. The behaviour of $|S_i|$ in the supercritical case is characterised in Lemma~\ref{lem:Si} (which is~\cite[Lemma 49(a--c)]{gm}). This changes significantly in the subcritical case, as below.

\begin{lemma}\lab{lsisub} There exist positive constants $B,Y_1,Y_2$ {dependent only on $r,k$,} and integers $I_0<I_1<B'$ as growing functions of $n$, such that \aas for every $ B\leq i< B'$:
\begin{enumerate}
\item[(a)] if $i\le I_0$, then $(1-Y_1\sqrt{|S_i|/n})|S_i|\leq |S_{i+1}|\leq  (1-Y_2\sqrt{|S_i|/n})|S_i|$ and $|S_i|=\Omega(n^{1-\d})$;
\item[(b)] if $i\ge I_1$, then $(1+Y_2\sqrt{|S_i|/n})|S_i|\leq |S_{i+1}|\leq  (1+Y_1\sqrt{|S_i|/n})|S_i|$ and $|S_i|=\Omega(n^{1-\d})$;
\item[(c)] $I_1-I_0=O(n^{\d/2})$ and for every $I_1\le i\le I_2$,
 $(1-Y_1 n^{-\d/2})|S_i|\leq |S_{i+1}|\leq  (1+Y_1n^{-\d/2})|S_i|$ and $|S_i|=\Theta(n^{1-\d})$;
\end{enumerate}
\end{lemma}

\begin{proof}[Proof of Lemma~\ref{lsisub}]

The proof of Lemma~\ref{lem:zeta} only requires that $|\hzeta_0-\zeta|<\eps$ for some sufficiently small constant $\eps>0$ (the value of $\eps$ depends on the gap between $\zeta$ and $r(k-1)$ in~\eqn{criticalZetaInequality}). Fix a small constant $\eps'>0$. By Lemma~\ref{lem:Hb}, there is a sufficiently large constant $B=B(r,k,\eps')>0$, such that the number of vertices in $H_B$ is at most $(\a+\eps')n$. By Theorem~\ref{tkim}(b), a.a.s.\ $|\zeta(H_B)-\zeta|=O(\eps')$. Hence, by choosing $\eps'>0$ sufficiently small (and thus $B$ sufficiently large), the evolution of $\theta_t$ is well depicted by Lemma~\ref{lem:zeta}, when SLOW-STRIP is applied to $H_B$. This fixes the constant $B$ in the lemma. Comparing with the analysis in Section~\ref{sec:Stripsub}, now we start our analysis from a configuration, $H_B$, that appears earlier than $G_0$ in SLOW-STRIP applied to $H$. Thus, all results in Section~\ref{sec:Stripsub} hold, except for a shift of the subscript $t$ in all parameters such as $\theta_t$.

Let $I_{\sigma}$ denote the first iteration in the parallel stripping process that the number of vertices becomes at most $\sigma n$.  
By Lemma~\ref{lem:tau3}, there is an integer $B'$ (corresponding to $I_2$ in Lemma~\ref{lem:tau3}), such that $L_{t(B')}\ge \eps n$ for all $t(B')\le t\le t(I_{\sigma}-1)$, provided that $\eps$ was chosen sufficiently small. This fixes the integer $B'$ in the lemma.

By Lemma~\ref{lem:D}, we can specify two iterations $I_0$ and $I_1$ (corresponding respectively to iterations 0 and $I_2$ in Lemma~\ref{lem:D}, due to a shift of the subscript) in the parallel stripping process such that
the value of $\theta(H_i)$ changes from $-\Theta(n^{-\d/2})$ to $\Theta(n^{-\d/2})$. Moreover, Lemma~\ref{lem:D} states that  a.a.s.\ $I_1-I_0=O(n^{\d/2})$, and both $L_{t(I_0)}$ and $L_{t(I_1)}$ are $\Theta(n^{1-\d})$. This fixes the integers $I_0$ and $I_1$ in the lemma and verifies the first claim of part (c).

It only remains to prove that $|S_i|$ changes in a rate described in parts (a,b,c). This part of the proof is similar to the approach in~\cite[Sections 6.2 and 6.3]{gm} ($\theta_t$ was denoted by $br_t$ in~\cite{gm}) so we briefly sketch the idea.

For (b): Consider $t(i)\le t\le t(i+1)$ for some $I_1\le i\le B'$. By Lemma~\ref{lem:D2} and Corollary~\ref{cor2:hit}, we may assume that $L_t=\Theta(t^2/n)$ and $\theta_t=\Theta(t/n)$. Thus, we may assume that $\theta_t=\Theta(\sqrt{L_t/n})$ for all $t$. Hence,
by~\eqn{Ldiff2}
\[
\ex(L_{t+1}-L_t\mid \calf_t)=\Theta(\sqrt{L_t/n}) +O\left(\frac{L_t}{n}\right),
\]
which is $\Theta(\sqrt{L_t/n})$ as long as $L_t/n$ bounded by some sufficiently small constant $\eps_1>0$. Then, by Lemma~\ref{l:azuma} we obtain the desired recursion in part (b), with $|S_i|$ replaced by $L_{t(i)}$, until $L_{t(i)}$ reaches $\eps_1 n$. We may choose $B'$ appropriately to ensure that a.a.s.\ $L_{t(B')}<\eps_1 n$ (e.g.\ choosing $B'$ to be the last step in the parallel stripping process after iteration $I_1$, such that the total degree of the light vertices is less than $\eps_1 n$). It is easy to show that a.a.s.\ $|S_i|=\Theta(|L_{t(i)}|)$ for every $i$ (e.g.\ see~\cite[eq.~(98)]{gm}), and immediately part (b) follows.

For (a): The analysis in Section~\ref{sec:Stripsub} only covers iterations in parts (b) and (c), as we started our analysis from iterations close to $I_0$ by the choice of $G_0$. However, the evolution of $L_{t(i)}$ for $B\le i\le I_0$  is ``symmetric'' to iterations from $I_1$ to $B'$. By Lemma~\ref{lem:zeta}, $\theta_t$ increases with a linear rate during all steps $t(B)\le t\le t(B')$. Then $L_t$ and $\theta_t$ for $t(B)\le t\le t(I_0)$ can be analysed in the same way as in Lemma~\ref{lem:D2}, except that $L_t$ decreases with a certain rate (rather than increases), due to the fact that $\theta_t$ is negative for $t(B)\le t\le t(I_0)$.

For (c):
Between iterations $I_0$ and $I_1$, we have in the proof of Lemma~\ref{lem:D} that $\theta_t=O(n^{-\d/2})$ and $L_t=\Omega(n^{1-\d})$.  Similar to the argument in part (b), part (c) follows by Lemma~\ref{l:azuma}. 

\end{proof}

In the subcritical case, a.a.s.\ the parallel stripping process (and SLOW-STRIP) terminates with an empty graph. Hence, a.a.s.\ every vertex is a non-$k$-core vertex. Let $\imax$ denote the last iteration of the parallel stripping process. By Theorem~\ref{thm:Stripsub}(a), we may assume that $\imax=O(n^{\d/2}\log n)$.

Let $B$ and $B'$ be integers in Lemma~\ref{lsisub}. 
We first prove that a.a.s.\ for every $v\in AP_r(n,cn)$, $|R^+(v)\cap H_{B'}|=O(\log n)$. Define $R_j(v)$ as in Definition~\ref{def:Rj}. We will actually show that $\cup_{j\ge B'} R_j(v)\supseteq R^+(v)\cap H_{B'}$ contains $O(\log n)$ vertices.
 We may assume that $v\in H_{B'}$ since otherwise $\cup_{j\ge B'} R_j(v)=\emptyset$. By Lemma~\ref{lem:tau3}, a.a.s.\ for all $B'\le i\le I_{\sigma}$, the total degree of $H_i$ decreases by $\Omega(n)$ in each iteration of the parallel stripping process. Therefore, a.a.s.\ $I_{\sigma}-B'=O(1)$. By Lemma~\ref{lem:small}, every component in $H_{I_\sigma+2}$ has size $O(\log n)$. It follows then that $|\cup_{j\ge I_{\sigma}+2} R_j(v)|=O(\log n)$. Now by Lemma~\ref{lem:neighbours}, $|\cup_{j\ge B'} R_j(v)|\le C\log n$ for some constant $C>0$, since all vertices in $\cup_{j\ge B'} R_j(v)$ are of distance $O(1)$ from $\cup_{j\ge I_{\sigma}+2} R_j(v)\subseteq H_{I_\sigma+2}$.

This allows us to focus on bounding $R_j(v)$ for $j\le B'$. 
The same proof of Lemma~\ref{lrec} yields the same bound for $D^-(R_j(v))$,  as below.
\be
D^-(R_j(v))\leq D^-(R_{j+1}(v))+Z\frac{|S_j|}{n}\sum_{\ell=i}^{j+1}D^-(R_{\ell}(v)) +\log^{14}n, \quad \mbox{for all}\ B\le j\le B'.\lab{rec2}
\ee
Moreover, we have shown that $|\cup_{j\ge B'}R_{j}(v)|\le C\log n$.


Let $i$ be the integer that $v\in S_i$. Next, we will recursively define $r_j$ such that
\[
D^-(R_{i-j}(v))\le r_j \quad \forall 0\le j\le i-B.
\]
This approach is similar to the work in~\cite[Section 5.5]{gm}. Note that~\eqn{rec2} holds only for $j\le B'$, and thus, we need to specify $r_j$ differently depending on whether $i\le B'$ or $i>B'$. To cope with that, define 
\bea
j_0&=&\max\{0,i-B'\}\\
r_j&=&n^{2\d} \quad \forall 0\le j\le j_0.
\eea
We first confirm that $D^-(R_{i-j}(v))\le r_j$ for all $0\le j\le j_0$. It is easy to show that a.a.s.\ the maximum degree of the original configuration $H\in AP_r(n,cn)$ is $O(\log n)$ (see, e.g.\ the proof of~\cite[Lemma 49(i)]{gm}). If $i\le B'$ then $j_0=0$, and so $D^-(R_i(v))=O(\log n)$, which is less than $r_0$. If $i>B'$, then $|\cup_{j\ge B'}R_{j}(v)|\le C\log n$ and so 
\be
\sum_{j=0}^{j_0}D^-(R_{i-j}(v))=O\left(\log n\sum_{j\ge B'}|R_{j}(v)|\right)=O(\log^2 n)\le r_j.\lab{sumj0}
\ee
Now, we have verified that $D^-(R_{i-j}(v))\le r_j$ for all $0\le j\le j_0$. Define
\bea
r_j&=&r_{j-1}+Z \frac{|S_{i-j}|}{n} \sum_{i=0}^{j-1} r_i+ n^{\d},\quad \forall j_0+1\le j\le i-B. \lab{recr}
\eea
Inductively, we prove that $D^-(R_{i-j}(v))\le r_j$ for all $j_0\le j\le i-B$. We have proved the base case. Assume it holds for $j-1$. Then, for $j$, by~\eqn{rec2} and induction,
\bean
D^-(R_{i-j}(v))&\leq& D^-(R_{i-j+1}(v))+Z\frac{|S_{i-j}|}{n}\sum_{\ell=0}^{j-1}D^-(R_{i-\ell}(v)) +\log^{14}n\\
&\le&r_{j-1}+Z\frac{|S_{i-j}|}{n}\sum_{\ell=j_0}^{j-1}r_{\ell}+Z\frac{|S_{i-j}|}{n}\sum_{\ell=0}^{j_0-1}D^-(R_{i-\ell}(v)) +\log^{14}n\le r_j,
\eean
by noting that $Z\frac{|S_{i-j}|}{n}\sum_{\ell=0}^{j_0-1}D^-(R_{i-\ell}(v)) +\log^{14}n=O(\log^{14}n)\le n^{2\d}$ by~\eqn{sumj0}.

Now, $r_j$ bounds $D^-(R_{i-j})\ge |R_{i-j}|$. It is convenient to define 
\[
t_j=\sum_{j_0\le\ell\le j} r_j.
\]
Then, again by~\eqn{sumj0},
\be
|\cup_{j\ge B} R_{j}(v)|\le \sum_{j=0}^{i-B}D^-(R_{i-j}(v))=\sum_{j=j_0}^{i-B}D^-(R_{i-j}(v))+\sum_{j=0}^{j_0-1}D^-(R_{i-j}(v))\le t_{i-B}+O(\log^2 n).\lab{t}
\ee
Noting that $r_j=t_j-t_{r-1}$,~\eqn{recr} yields the following recurrence for $t_j$:
\be
t_{j}-t_{j-1}= t_{j-1}-t_{j-2}  + Z\frac{|S_{i-j}|}{n}t_{j-1}
+ n^{2\d}, \quad \forall j\ge j_0+1,\lab{rec2}
\ee
where $t_{j_0}=r_{j_0}=n^{2\d}$ and $t_{j_0-1}=0$.
Same as in~\cite{gm}, we can find a sequence $(a_j,b_j)_{j\ge j_0}$ such that $a_{j_0}=b_{j_0}=1$, $b_j\le 1$ and $a_j\le 1+D\sqrt{|S_{i-j}|/n}$ for some constant $D>0$ such that
\be
t_j-a_jt_{j-1}=b_j(t_{j-1}-a_{j-1}t_{j-2})+n^{2\d},\quad \forall j\ge j_0+1.\lab{ezzzz}
\ee
See~\cite[eq.~(64) and (65)]{gm} for the detailed construction of the sequence $(a_j,b_j)$.

Let $c_j=t_j-a_jt_{j-1}$.
Then~\eqn{rec2} becomes
\[
c_{j}=b_j c_{j-1}+n^{2\d} \le c_{j-1}+n^{2\d}\le c_{0} +j n^{2\d}=r_0+j n^{2\d}.
\]
Since $r_0=n^{2\d}$ and $j\le \imax=O(n^{\d/2}\log n)$,
we have
\be
t_j-a_jt_{j-1}\le U:= n^{3\d}, \quad \forall j\ge j_0+1.\lab{rec3}
\ee
So far we have deduced a recurrence for $t_j$, which is the same as in~\cite[eq.~(66)]{gm}. The bound on $t_j$ depends on the sequence $(a_j)$, which is a function of $|S_j|/n$. Since the evolution of $|S_j|$ is different in the subcritical case (comparing Lemma~\ref{lsisub} with Lemma~\ref{lem:Si}), the analysis is a little different. Our eventual goal is to bound $t_{i-B}$ by $n^{O(\d)}$ by recursively bounding each $t_j$, $j\le i-B$. We will break the analysis into three different stages and bound $t_{i-I_1}$, $t_{i-I_0}$ and $t_{i-B}$ in turn, where $I_0$ and $I_1$ are the integers stated in Lemma~\ref{lsisub}.\smallskip

{\em Stage 1: bounding $t_{i-I_1}$}. We may assume that $i>I_1$; otherwise we may skip this stage.
Recursively applying~\eqn{rec3} for all $j_0+1\le j<i-I_1$, we have
\be
t_{i-I_1}\le U\left(1+\sum_{h=j_0+2}^{i-I_1}\prod_{j=h}^{i-I_1}a_j\right)+ t_{j_0}\prod_{j=j_0+1}^{i-I_1}a_j.\lab{rec5}
\ee
Our goal is to bound $t_{i-I_1}$ by $n^{O(\delta)}$.
We first show that $\prod_{j=j_0+1}^{i-I_1}a_j=n^{O(\d)}$ (this part of the analysis is similar to the work in~\cite{gm}; see~\cite[eq.~(69)--(70)]{gm}).

Since $a_j\le 1+D\sqrt{|S_{i-j}|/n}$ for each $j$, we have
\be
\prod_{j=j_0+1}^{i-I_1}a_j\le \exp\left(D\sum_{j=j_0+1}^{i-I_1}\sqrt{\frac{|S_{i-j}|}{n}}\right)=\exp\left(D\sum_{j=I_1}^{i-j_0-1}\sqrt{\frac{|S_{j}|}{n}}\right)\le \exp\left(D\sum_{j=I_1}^{B'-1}\sqrt{\frac{|S_{j}|}{n}}\right),\lab{rec4}
\ee
where the last inequality above holds by the definition of $j_0$. By Lemma~\ref{lsisub}(b),
\[
|S_{B'}|\ge |S_{I_1}|\prod_{j=I_1}^{B'-1}\left(1+Y_2\sqrt{\frac{|S_j|}{n}}\right)\ge |S_{I_1}|\exp\left(Y_3\sum_{j=I_1}^{B'-1}\sqrt{\frac{|S_j|}{n}}\right),
\]
for some constant $Y_3>0$, as $|S_j|/n$ is small for every $j$ in this range. Again, by Lemma~\ref{lsisub}(b), $|S_{I_1}|\ge C_1 n^{1-\d}$ for some constant $C_1>0$. This implies that, using $|S_{B'}|\le n$,
\[
\exp\left(Y_3\sum_{j=I_1}^{B'-1}\sqrt{\frac{|S_j|}{n}}\right)\le \frac{|S_{B'}|}{|S_{I_1}|}\le \frac{n^{\d}}{C_1}.
\]
Substituting this into~\eqn{rec4}, we have
\[
\prod_{j=j_0+1}^{i-I_1}a_j\le \left(\frac{n^{\d}}{C_1}\right)^{D/Y_3}=n^{O(\d)}.
\]
Now~\eqn{rec5} gives
\be
t_{i-I_1}=O\left(U \imax\prod_{j=j_0+1}^{i-I_1}a_j\right) +t_{j_0} \prod_{j=j_0+1}^{i-I_1}a_j = n^{O(\d)}. \lab{stage1}
\ee
\smallskip

{\em Stage 2: bounding $t_{i-I_0}$}. Using the same recursion~\eqn{rec3} for $j$ such that $I_0\le i-j< I_1$,
we have
\be
t_{i-I_0}\le U\left(1+\sum_{h=i-I_1+2}^{i-I_0}\prod_{j=h}^{i-I_0}a_j\right)+ t_{i-I_1}\prod_{j=i-I_1+1}^{i-I_0}a_j\le (U\imax+t_{i-I_1}) \prod_{j=i-I_1+1}^{i-I_0}a_j.\lab{rec55}
\ee
Same as before, we have
\[
\prod_{j=i-I_1+1}^{i-I_0}a_j\le \exp\left(D\sum_{j=i-I_1+1}^{i-I_0}\sqrt{\frac{|S_{i-j}|}{n}}\right)=\exp\left(D\sum_{j=I_0}^{I_1-1}\sqrt{\frac{|S_{j}|}{n}}\right).
\]
By Lemma~\ref{lsisub}(c), there are constants $C_2,C_3>0$ such that
$|S_j|/n\le C_2n^{-\d}$ for all $I_0\le j\le I_1-1$ and $I_1-I_0\le C_3 n^{\d/2}$. This implies that
\[
\exp\left(D\sum_{j=I_0}^{I_1-1}\sqrt{\frac{|S_{j}|}{n}}\right)\le \exp(D C_3 \sqrt{C_2}),
\]
and thus, $\prod_{j=i-I_1+1}^{i-I_0}a_j=O(1)$.
This together with~\eqn{rec55} and~\eqn{stage1} implies
\be
t_{i-I_0}=n^{O(\d)}.\lab{stage2}
\ee

{\em Stage 3: bounding $t_{i-B}$}. Using~\eqn{rec3} for $j$ such that $B\le i-j< I_0$, we have
\be
t_{i-B}\le U\left(1+\sum_{h=i-I_0+2}^{i-B}\prod_{j=h}^{i-B}a_j\right)+ t_{i-I_0}\prod_{j=i-I_0+1}^{i-B}a_j\le (U\imax+t_{i-I_0})\prod_{j=i-I_0+1}^{i-B}a_j.\lab{rec6}
\ee
By Lemma~\ref{lsisub}(a),
\[
|S_{I_0}|\le |S_{B}|\prod_{j=B}^{I_0-1}\left(1-Y_2\sqrt{\frac{|S_j|}{n}}\right)\le |S_{B}|\exp\left(-Y_2\sum_{j=B}^{I_0-1}\sqrt{\frac{|S_j|}{n}}\right).
\]
 By Lemma~\ref{lsisub}(a), $|S_{I_0}|\ge C_4 n^{1-\d}$ for some constant $C_4>0$. This implies that, using $|S_B|\le n$,
\[
\exp\left(Y_2\sum_{j=B}^{I_0-1}\sqrt{\frac{|S_j|}{n}}\right)\le \frac{|S_B|}{|S_{I_0}|}\le \frac{n^{\d}}{C_4}.
\]
This gives
\[
\prod_{j=i-I_0+1}^{i-B}a_j\le \exp\left(D\sum_{j=B}^{I_0-1}\sqrt{\frac{|S_{j}|}{n}}\right) \le \left(\frac{n^{\d}}{C_4}\right)^{D/Y_2}=n^{O(\d)}.
\]
This together with~\eqn{stage2} and~\eqn{rec6} yields
\be
t_{i-B}=n^{O(\d)}. \lab{stage3}
\ee

Now, we have shown, by~\eqn{t}, that $\sum_{j\ge B}|R_j(v)|\le t_{i-B}+O(\log^2 n)=n^{O(\d)}$. It follows immediately that $\sum_{j\ge 0}|R_j(v)|=n^{O(\d)}$ by Lemma~\ref{lem:neighbours}. Hence, noting that $|R^+(v)|\le \cup_{j\ge 0} |R_j|$, we have shown that a.a.s.\ $|R^+(v)|=n^{O(\d)}$ for every $v\in AP_r(n,cn)$. This proves  Theorem~\ref{depth2} and therefore Theorem~\ref{thm:depthsub}, by Corollary~\ref{ccon}. \qed

\end{document}